\documentclass[11pt,english]{article}

\usepackage[english]{babel}    
\usepackage{amsmath}           
\usepackage[utf8]{inputenc}    
\usepackage[T1]{fontenc}       
\usepackage{longtable}         
\usepackage{exscale}           
\usepackage[final]{graphicx}   
\usepackage[sort]{cite}        
\usepackage{eucal}
\usepackage{array}             
\usepackage[a4paper]{geometry} 
\usepackage{xspace}            
\usepackage{tikz-cd}
\usepackage{amssymb}           
\usepackage{bbm}               
\usepackage{mathtools}         
\usepackage[amsmath,thmmarks,hyperref]{ntheorem} 
\usepackage{mathrsfs}          
\usepackage[scr]{rsfso}        
\usepackage{paralist}          
\usepackage{aliascnt}          
\usepackage[expansion=false    
]{microtype}        
\usepackage[nottoc]{tocbibind} 
\usepackage[backref=page,     
final=true,         
pdfpagelabels       
]{hyperref}         

\usepackage[colorinlistoftodos]{todonotes}

\numberwithin{equation}{section}

\allowdisplaybreaks

\renewcommand{\arraystretch}{1.2}

\let\originalleft\left
\let\originalright\right
\renewcommand{\left}{\mathopen{}\mathclose\bgroup\originalleft}
\renewcommand{\right}{\aftergroup\egroup\originalright}

\theoremheaderfont{\normalfont\bfseries}
\theorembodyfont{\itshape}
\newtheorem{lemma}{Lemma}[section]

\newaliascnt{proposition}{lemma}
\newtheorem{proposition}[proposition]{Proposition}
\aliascntresetthe{proposition}

\newaliascnt{thm}{lemma}
\newtheorem{theorem}[thm]{Theorem}
\aliascntresetthe{thm}

\newaliascnt{corollary}{lemma}
\newtheorem{corollary}[corollary]{Corollary}
\aliascntresetthe{corollary}

\newaliascnt{definition}{lemma}
\newtheorem{definition}[definition]{Definition}
\aliascntresetthe{definition}

\newaliascnt{claim}{lemma}

\aliascntresetthe{claim}

\theorembodyfont{\rmfamily}

\newaliascnt{example}{lemma}
\newtheorem{example}[example]{Example}
\aliascntresetthe{example}

\newaliascnt{remark}{lemma}
\newtheorem{remark}[remark]{Remark}
\aliascntresetthe{remark}

\newaliascnt{conjecture}{lemma}
\newtheorem*{conjecture}[conjecture]{Conjecture}
\aliascntresetthe{conjecture}

\theorembodyfont{\itshape}
\theoremnumbering{Roman}
\newtheorem{maintheorem}{Main Theorem}

\makeatletter
\def\theorem@checkbold{}
\makeatother

\theoremheaderfont{\scshape}
\theorembodyfont{\normalfont}
\theoremstyle{nonumberplain}
\theoremseparator{:}
\theoremsymbol{\hbox{$\heartsuit$}}
\newtheorem{proof}{Proof}
\theoremsymbol{\hbox{$\triangledown$}}

\newenvironment{propositionlist}{\begin{compactenum}[\itshape i.)]}{\end{compactenum}}
\newenvironment{definitionlist}{\begin{compactenum}[\itshape i.)]}{\end{compactenum}}

\newcommand{\cc}[1]          {\overline{{#1}}}

\newcommand{\D}              {\mathop{}\!\mathrm{d}}
\newcommand{\del}{\mathop{}\!\partial}
\newcommand{\lie}[1]          {\mathfrak{#1}}
\DeclareMathOperator{\Ad}     {\mathrm{Ad}}
\newcommand{\Fun}[1][k]      {\mathscr{C}^{#1}}
\newcommand{\Cinfty}         {\Fun[\infty]}
\newcommand{\Sec}[1][k]      {\Gamma^{#1}}
\newcommand{\Secinfty}       {\Sec[\infty]}

\newcommand{\at}[1]          {\big|_{#1}}

\newcommand{\pr}             {\mathrm{pr}}
\newcommand{\Lie}   {\mathscr{L}}
\DeclarePairedDelimiter{\abs}{\lvert}{\rvert}
\newcommand{\tensor}[1][{}]           {\mathbin{\otimes_{\scriptscriptstyle{#1}}}}

\newcommand{\Anti}                    {\Lambda}
\newcommand{\Sym}                     {\mathrm{S}}

\newcommand{\poly}        {\mathrm{poly}}
\newcommand{\Der}   {\mathrm{Der}}
\newcommand{\algebra}[1] {\mathscr{#1}}
\DeclareMathOperator{\id}    {\mathsf{id}}
\newcommand{\argument}       {\,\cdot\,}
\newcommand{\red}{\mathrm{red}} 
\newcommand{\comm}{\mathrm{Com}}
\newcommand{\cocom}{\mathrm{coCom}}
\newcommand{\ucocom}{\mathrm{ucoCom}}
\newcommand{\flie}{\mathrm{Lie}}
\newcommand{\colie}{\mathrm{coLie}}
\newcommand{\gerst}{\mathrm{Gerst}}

\newcommand{\multivect} {\mathscr{V}}
\newcommand{\sh}             {\mathrm{sh}}
\newcommand{\End}            {\operatorname{\mathrm{End}}}
\newcommand{\acts}            {\mathbin{\triangleright}}

\title{Quantization of the Momentum Map via $\frak{g}$-adapted Formalities}
\date{}

\author{
	\textbf{Chiara Esposito}\thanks{\texttt{chesposito@unisa.it}},\\[0.3cm]
	Dipartimento di Matematica\\
	Università degli Studi di Salerno\\
	Italy \\[0.5cm]
	\textbf{Ryszard Nest}\thanks{\texttt{rnest@math.ku.dk}},\\[0.3cm]
	Department of Mathematical Sciences\\
	University of Copenhagen\\
	Denmark\\[0.5cm]
	\textbf{Jonas Schnitzer}\thanks{\texttt{jonaschristoph.schnitzer@unipv.it }},\\[0.3cm]
	Dipartimento di Matematica ``Felice Casorati''\\
	Università degli Studi di Pavia\\
	Italy\\[0.5cm]
	\textbf{Boris Tsygan}\thanks{\texttt{b-tsygan@northwestern.edu}},\\[0.3cm]
	Department of Mathematics\\
	Northwestern University\\
	USA
}

\begin{document}
	
\maketitle

\begin{abstract}
	In this note, we provide a proof of the existence and complete classification of $G$-invariant star products with quantum momentum maps
	on Poisson manifolds by means of an equivariant version of the formality theorem. 
\end{abstract}

\newpage

\tableofcontents

\newpage

%
%
\section{Introduction}
\label{sec:Introduction}

This paper provides a proof of the equivariant version of the formality
theorem, that implies the existence and classification of quantum momentum maps.  The topic of the quantization of the momentum map has been widely studied in various settings,
e.g. \cite{bieliavsky:2020,fedosov:1998,gutt:2003,hamachi:2002,lu:1993,muller:2004,reichert:2015}.
Our motivation comes from formal deformation quantization. Deformation
quantization was introduced by Bayen, Flato, Fronsdal,
Lichnerowicz and Sternheimer in \cite{bayen.et.al:1977a, bayen.et.al:1978a} and it relies
on the idea that the quantization of a phase space described by a
Poisson manifold $M$ is described by a formal deformation, so-called
\emph{star product}, on the commutative algebra of smooth
complex-valued functions $\Cinfty(M)$ in a formal parameter $\hbar$.
The existence and classification of star products on Poisson manifolds
has been provided by Kontsevich's formality theorem
\cite{kontsevich:2003a} and extended to the invariant setting of Lie group
actions in \cite{dolgushev:2005a,dolgushev:2005}. More explicitly, the
formality theorem provides an $L_\infty$-quasi-isomorphism between the
differential graded Lie algebras (DGLA) of polyvector fields and the
polydifferential operators resp. the invariant versions.  As such, it
maps in particular Maurer--Cartan elements in the DGLA of polyvector
fields, i.e. (formal) Poisson structures, to Maurer-Cartan elements in
the DGLA of polydifferential operators, which correspond to star
products.

One open question and our main motivation is to investigate the
compatibility of deformation quantization and phase space reduction in
the Poisson setting.
At the classical level conserved quantities described via the
momentum map lead to phase space reduction which constructs
from the high-dimensional original phase space one of a
smaller dimension. Motivated by the significance of the
classical situation, it is highly desirable to find an
analogue in quantum physics.  

In the classical setting one considers here the Marsden-Weinstein reduction 
\cite{marsden.weinstein:1974a}. Suppose that a Lie group 
$G$ acts by Poisson diffeomorphisms on the Poisson manifold $M$ 
and that it allows an $\Ad^*$-equivariant momentum map 
$J\colon M\longrightarrow \lie g^*$ with $0\in \lie g^*$
as regular value, where $\lie g$ is the Lie algebra of
$G$. Then $C = J^{-1}(\{0\})$ is a closed embedded submanifold of
$M$ and the reduced manifold $M_\red = C/G$ is again a Poisson 
manifold under suitable assumptions. Reduction theory is very important
and it is still a very active field of research. Among the others, we mention the categorical reformulation performed in \cite{dippell:2019a}.
In the setting of deformation quantization a quantum reduction scheme
has been introduced in \cite{bordemann.herbig.waldmann:2000a},
see also \cite{gutt.waldmann:2010a} for a slightly different formulation,
which allows to study the compatibility between the reduction scheme 
and the properties of the star product, as in \cite{esposito:2019a}.
Given a Lie group 
$G$ acting on a manifold $M$, one can consider \emph{equivariant star products} $(\star,H)$, i.e. 
pairs consisting of an invariant star product $\star$ and a 
quantum momentum map $H= \sum_{r=0}^\infty \hbar^r J_r \colon \lie g \longrightarrow \Cinfty(M)[[\hbar]]$, where $\lie g$ is the Lie algebra of
$G$. In this case, $J_0$ is a classical momentum map for the Poisson structure induced by $\star$. 
In this setting, there is a well-known BRST-like reduction procedure \cite{bordemann.herbig.waldmann:2000a} of equivariant star products on $M$ to star products on $M_\red$.
Nevertheless, the existence and classification of equivariant star products have remained an open question for many years.

Recently, both classical and quantum reduction schemes have been obtained for polydifferential
operators and polyvector fields phrased in terms of
$L_\infty$-morphisms 
\cite{esposito:2022,esposito:2023}. 
Classical and quantum momentum
maps can be encoded by the complexes of equivariant
multivector fields and multidifferential operators,
resp.   To prove existence and classification of equivariant
star products the idea  consists of
finding an $L_\infty$-morphism between the equivariant
complexes, that sends a classical into quantum momentum
map. More precisely, as we can interpret the classical and
quantum momentum map as Maurer--Cartan elements of these
$L_\infty$-algebras, the existence of an $L_\infty$-morphism
automatically sends classical into quantum momentum maps.
However, it turns out that the equivariant complexes are not
DGLA's but they are \emph{curved Lie algebras}, see
\cite{esposito:2022,esposito:2023,kraft:2024}, where the
curvature is given by the infinitesimal generators of the
action. Since we do not have a differential, we do not have a
notion of quasi-isomorphism of curved Lie algebras.  A good
replacement for quasi-isomorphisms in this picture are
homotopy equivalences of curved Lie algebras, which boils down to 
quasi-isomorphisms if the curvature vanishes. Moreover, homotopy equivalent
morphisms induce the same maps on the corresponding
deformation spaces.
The
compatibility with the curvature turns out to be crucial. In facts, we prove that an $L_\infty$-morphism between
$L_\infty$-algebras can be lifted to a curved
$L_\infty$-morphism if and only if it is compatible with the
curvature. Due to its importance we turn this condition into a
definition and introduce the so-called
\emph{$\mathfrak{g}$-adapted $L_\infty$-algebras}.

The strategy to construct a curved $L_\infty$ morphism consists of the following steps: first
we construct a flat
$L_\infty$-morphism $F\colon T_\poly(M)\to D_\poly$ by means of a refinement of a diagram from \cite{dolgushevetal:2007} of the form
\begin{center}
\begin{tikzcd}
T_\poly (M) & \bullet\arrow[l] & \bullet\arrow[l]\arrow[r] & \bullet & D_\poly(M)\arrow[l],
\end{tikzcd}
\end{center}
which is a zig-zag of
$L_\infty$-quasi-isomorphisms. 
It is important to remark that we also get a concrete quasi-isomorphism. As a second step we prove that the resulting curved morphism is also $\lie{g}$-adapted. 
More precisely, we obtain the following 
\begin{maintheorem}
	Under suitable 
	assumptions on $G$ and $M$, there exists a $\lie{g}$-adapted
    $L_\infty$-morphism
	\begin{align*}
		F\colon T_\poly(M)\to D_\poly(M),
	\end{align*}
	such that the
	first Taylor component $F_1$ coincides with the Hochschild-Kostant-Rosenberg map $\mathrm{hkr}$.
\end{maintheorem}
As a second step we show that this condition is enough to lift the$L_\infty$- morphism $F$ to a $L_\infty$-morphism
between equivariant multivector fields and equivariant multidifferential operators, which immediately
implies our second main achievement.
\begin{maintheorem}
	Under the same assumptions of the above theorem, there is a one-to-one correspondence
    between equivalence classes of equivariant star
    products and equivalence
    classes of $G$-invariant formal Poisson structures
    with formal momentum maps.
\end{maintheorem}
The paper is organized as follows: Section~\ref{sec:Preliminaries} collects all the basic notions of $L_\infty$-algebras and, in
particular, curved Lie algebras, as well as some technical tools from homological algebra. In Section~\ref{sec:zigzag} 
we show how to construct a $G_\infty$-quasi isomorphism connecting the Hochschild complex and 
the DGLA of multivector fields on a generic algebra. Section~\ref{sec:gadapted} contains the definitions of the equivariant complexes
and their interpretation to their Maurer–Cartan elements. We discuss a characterization of curved
$L_\infty$-morphism and we introduce the notion of $\lie g$-adapted $L_\infty$-algebras. This allows us to prove under
which condition, given a $\lie g$-adapted $L_\infty$-quis, its quasi-inverse is again $\lie g$-adapted. Finally, in Section~\ref{sec:formality}
we construct the $\lie g$-adapted quasi-inverse we look for.

\section*{Acknowledgments}
The authors are grateful to Alexander
Berglund for the helpful comments. Moreover, we are thankful to Andreas 
Kraft who contributed ideas to the bigger context of this work at earlier stages. 
C.E. was supported by the
National Group for Algebraic and Geometric Structures, and
their Applications (GNSAGA -- INdAM). 

This paper was largely developed in the inspiring environment of the MFO Institute, 
made possible by the Oberwolfach Research Fellows program. 
We are deeply grateful to the local staff, as well as to the many researchers
we encountered during our stay, 
with special thanks to Cosima, Daniele, Jitendra, Pedro, Rosa, and Severin.
And then, there were the alpacas -- majestic and surprisingly wise creatures 
whose calming aura and unwavering tranquility sparked our creativity 
with their comforting presence in ways we never expected.

%
%
\section{Preliminaries}
\label{sec:Preliminaries}
In this section we recall some basic notions, which are crucial to all further discussions, 
to fix notations and conventions.

\subsection{Homological Perturbation}
\label{sec:HomologicalPerturbation}
We first recall the classical homological perturbation lemma, see e.g. \cite{crainic:2004a:pre, huebschmann:1991}.  
Note that we have chosen different sign conventions.
\begin{definition}[Homotopy retract]
	Let $(C^\bullet,\D_C)$ and $(D^\bullet,\D_D)$ two cochain complexes. We call homotopy retract
	a diagram of the form:
	\begin{equation}
		\label{eq:homretract}
		\begin{tikzcd}
			(C^\bullet,\D_C)
			\arrow[r," i  ", shift left = 3pt]
			&(D^\bullet,\D_D)
			\arrow[l," p ", shift left = 3pt]
			\arrow[loop,
			out = -30,
			in = 30,
			distance = 30pt,
			start anchor = {[yshift = -7pt]east},
			end anchor = {[yshift = 7pt]east},
			" h  "{swap}
			]
		\end{tikzcd},
	\end{equation}
	where $i$ and $p$ are cochain maps and $h$ is a degree $-1$ map, such that  $\id-ip=\D_Dh+h\D_D$.
\end{definition}
When there is no ambiguity about the differentials, we refer to the retract simply by $(i, p, h)$.
\begin{definition}[Perturbation]
	\label{def:Perturbation}%
	Let $(i, p, h)$ be a homotopy retract.
	\begin{definitionlist}
		\item \label{item:Perturbation} A $ R $-linear map
		$ b   \colon D^\bullet \to D^{\bullet+1}$ such that
		$(\D_D +  b  )^2 = 0$ is said a \emph{perturbation} of
		$( i  , p , h  )$.
		\item \label{item:SmallPerturbation} A perturbation $ b  $ is
		called \emph{small} if $\id +  b    h  $ is invertible.
		\item \label{item:LocallyNilpotentPerturbation} A perturbation
		$ b  $ is called \emph{locally nilpotent} if for every
		$a \in D^\bullet$ there exists $n \in \mathbb{N}_0$ such that
		$( b    h  )^n(a) = 0$.
	\end{definitionlist}
\end{definition}
\begin{equation}
	\label{eq:LocallyNilpotentInverse}
	(\id +  b   h  )^{-1}
	=
	\sum_{n=0}^{\infty} (-1)^n ( b   h  )^n.
\end{equation}
The homological perturbation lemma proves that given a homotopy retract 
together with a small perturbation, one can immediately obtain
a new retract (see \cite{crainic:2004a:pre}). More precisely, we have the following
proposition.
\begin{proposition}[Homological perturbation]
	\label{prop:HomologicalPerturbation}%
	Let $(i, p, h)$ be a homotopy retract given by the diagram \eqref{eq:homretract} and let $b$
	be a small perturbation. Then the perturbed diagram
	\begin{equation}
		\label{eq:perturbed}
		\begin{tikzcd}
			(C^\bullet, \mathrm{D}  _C)
			\arrow[r," I  ", shift left = 3pt]
			&(D^\bullet,\D_D +  b  )
			\arrow[l," P ", shift left = 3pt]
			\arrow[loop,
			out = -30,
			in = 30,
			distance = 30pt,
			start anchor = {[yshift = -7pt]east},
			end anchor = {[yshift = 7pt]east},
			" H  "{swap}
			]
		\end{tikzcd}
	\end{equation}
	is again a homotopy retract with
	\begin{alignat}{2}
		\mathrm{D}  _C &\coloneqq
		\D_C
		+  p (\id +  b   h  )^{-1} b    i  ,
		\qquad\qquad
		& I   &\coloneqq
		(\id +  h    b  )^{-1} i  ,
		\\
		H   &\coloneqq
		(\id +  h   b  )^{-1} h  ,
		& P  &\coloneqq
		p (\id +  b   h  )^{-1}.
	\end{alignat}
\end{proposition}
\begin{corollary}[Morphism of perturbed retracts]
	\label{cor:MorphismOfPerturbations}%
	Let $\Phi \colon D^\bullet \to D'^\bullet$ be a morphism of
	homotopy retracts $( i  , p , h  )$ and
	$( i', p ', h  ')$.  Moreover, let $ b  $ be a small
	perturbation of $( i  , p , h  )$ and let $ b  '$ be small
	perturbation of $( i  ', p ', h  ')$.  If $\Phi$ satisfies
	$ b  '\Phi = \Phi  b  $, then $\Phi$ is a morphism between the
	perturbed homotopy retracts, too.
\end{corollary}
This corollary still holds if we consider morphisms of homotopy
retracts along a ring homomorphism.  
\begin{definition}[Special deformation retract]
	\label{def:SpecialDeformationRetract}%
	Let $(i, p, h)$ be a homotopy retract given by the diagram \eqref{eq:homretract}
	\begin{definitionlist}
		\item \label{item:DeformationRetract} If additionally
		\begin{equation}
			\label{eq:DeformationRetract}
			p  i   = \id
		\end{equation}
		holds, the above homotopy retract is called \emph{deformation
			retract}.
		\item \label{item:SpecialDeformationRetract} A deformation retract
		such that 
		\begin{equation}
			h  ^2 = 0,
			\qquad
			p   h   = 0,
			\quad
			\textrm{ and }
			\quad
			h    i   = 0
		\end{equation}
		hold is called a \emph{special deformation retract}.
	\end{definitionlist}
\end{definition}
Note that Equation~\eqref{eq:DeformationRetract} is equivalent to
\begin{equation}
	\begin{tikzcd}
		(D^\bullet,\D_D)
		\arrow[r," p ", shift left = 3pt]
		&(C^\bullet,\D_C)
		\arrow[l," i  ", shift left = 3pt]
		\arrow[loop,
		out = -30,
		in = 30,
		distance = 30pt,
		start anchor = {[yshift = -7pt]east},
		end anchor = {[yshift = 7pt]east},
		"0"{swap}
		]
	\end{tikzcd}
\end{equation}
being a homotopy retract.
Perturbations of deformation retracts will in general not be deformation retracts.
As an immediate consequence of Prop.~\ref{prop:HomologicalPerturbation}
one can prove that special deformation retracts are preserved under perturbation
\cite{crainic:2004a:pre}.
\begin{corollary}[Perturbation of special deformation retract]
	\label{cor:SpecialDeformationRetract}%
	Let $(i, p, h)$ be a special deformation retract
	together with a locally nilpotent perturbation $ b  $.
	Then the perturbed one \eqref{eq:perturbed}
	 is again a special
	deformation retract.
\end{corollary}

\subsection{$L_\infty$-algebras}
$L_\infty$-algebras were first introduced in \cite{lada.markl:1994a} and 
\cite{lada.stasheff:1993a} and play a prominent role in deformation theory.
In this section we recall the notions of $L_\infty$-algebras,
$L_\infty$-morphisms and their twists by Maurer--Cartan elements. 
For proofs and details we refer the reader to
\cite{kraft:2024}.

We denote by $V^\bullet$ a graded vector space over a field
$\mathbb{K}$ of characteristic $0$ and define the \emph{shifted}
vector space $V[k]^\bullet$ by
\begin{equation*}
	V[k]^\ell
	=
	V^{\ell+k}.
\end{equation*}
\begin{definition}[(Curved) $L_\infty$-algebra]
A degree $+1$ coderivation $Q$ on the coaugmented counital conilpotent
cocommutative coalgebra $\ucocom(\mathscr{L}[1])$ cofreely cogenerated by the
graded vector space $\mathscr{L}[1]^\bullet$ over $\mathbb{K}$ is
called a \emph{curved $L_\infty$-structure} on the graded vector space
$\mathscr{L}$ if $Q^2=0$.
If $Q(1)=0$ we simply say $L_\infty$-algebra.
\end{definition}
Here 
$\ucocom(V)$ denotes the counital conilpotent cocommutative coalgebra cogenerated by a (graded) vector space 
$V$, which 
can be realized as the symmetrized deconcatenation
coproduct on the space $\bigoplus_{n\geq0}\bigvee^n\mathscr{L}[1]$
where $\bigvee^n\mathscr{L}[1]$ is the space of coinvariants for the
usual (graded) action of the symmetric group in $n$ letters $S_n$ on
$\otimes^n\mathscr{L}[1]$, see e.g. \cite{esposito:2021}. Let us denote by 
\begin{equation}
	Q_n^k\colon \bigvee^n(\mathscr{L}[1])\longrightarrow \bigvee^k\mathscr{L}[2]
\end{equation}
the components of the coderivation. 
 Any degree $+1$
coderivation $Q$ on $\ucocom(\mathscr{L})$ is then uniquely determined by the
components $Q^1_n$
through the formula 
\begin{equation}
	Q(\gamma_1\vee\ldots\vee\gamma_n)
	=
	\sum_{k=0}^n\sum_{\sigma\in\mbox{\tiny Sh($k$,$n-k$)}}
	\epsilon(\sigma)Q_k^1(\gamma_{\sigma(1)}\vee\ldots\vee
	\gamma_{\sigma(k)})\vee\gamma_{\sigma(k+1)}\vee
	\ldots\vee\gamma_{\sigma(n)}.
\end{equation} 
Here Sh($k$,$n-k$) denotes the set of $(k, n-k)$ shuffles in $S_n$, and
$\epsilon(\sigma)=\epsilon(\sigma,\gamma_1,\ldots,\gamma_n)$ is a sign
given by the rule $
\gamma_{\sigma(1)}\vee\ldots\vee\gamma_{\sigma(n)}=
\epsilon(\sigma)\gamma_1\vee\ldots\vee\gamma_n $.
We use the
conventions that Sh($n$,$0$)=Sh($0$,$n$)$=\{\id\}$ and that the empty
product equals the unit. The condition $Q(1)=0$ is equivalent to the vanishing of $Q_0$.  
In the following we are interested in two particular $L_\infty$-algebras that we briefly
recall.
\begin{example}[Multivector fields]
	Let $(A,\mu_0)$ be an commutative algebra. We define
	the graded vector space
	\begin{align}
		\multivect^n(A)
		:=
		\Sym_A\Der(A)[-1]
	\end{align}
which is a Gerstenhaber algebra via product of the symmetric algebra and the extension of the commutator. 
Here, $\Sym_A$ denotes the symmetric algebra over the commutative algebra $A$. 

For a smooth manifold 
$M$ we write $\multivect(M)$ instead of $\multivect(\Cinfty(M))$. Moreover, we follow the literature and denote by $T_\poly(M)=\multivect(M)[1]$ the associated graded Lie algebra of multivector fields.  
\end{example}
\begin{example}[Multidifferential operators]
	Let $(A,\mu_0)$ be an associative algebra. We define
	the graded vector space
	\begin{equation}
		C^n(A)
		:=
		\mathrm{Hom}(A^{\otimes n},A)
	\end{equation}
	for $n \geq 1$ with $C^{0}(A)=A$. 
	Let $D,E\in C^\bullet (A)$ and $d=\abs D$, $e =\abs E$.
	We define the concatenation of 
	$D,E$ by 
	\begin{equation}
		D \circ E (a_1,\dots,a_{d+e-1}) 
		= 
		\sum_{i=0}^{|D|-1} (-1)^{i (|E|-1)} 
		D(a_0,\dots, a_{i-1}, E(a_i,\dots,a_{i+e}),a_{i+e+1},\dots,a_{d+e})
	\end{equation}
	for any $a_1,\dots,a_{d+e-1}\in A$. As an immediate consequence, we define the \emph{Gerstenhaber bracket}
	as graded commutator with respect to the concatenation:
	\begin{equation}
		\label{eq:GerstenhaberBracketClassical}
		[D,E]_G
		=
		(-1)^{(\abs{E}-1)(\abs{D}-1)} \left(D \circ E - (-1)^{(\abs{E}-1)(\abs{D}-1)} E \circ D\right).
	\end{equation}
	Note that we use the sign convention from \cite{bursztyn.dolgushev.waldmann:2012a}.
	This yields a graded Lie algebra
	$(C^{\bullet}(A)[1],[\argument{,}\argument]_G)$.  The product $\mu_0$
	is an element of $C^1(A)[1]$ and one notices that the associativity of
	$\mu_0$ is equivalent to $[\mu_0,\mu_0]_G=0$.  Furthermore, we get an
	induced differential $\del \colon C^\bullet(A) \rightarrow C^{\bullet
		+1}(A)$ via $\del = [\mu_0,\argument]_G$, the so-called
	\emph{Hochschild differential}, and thus a DGLA structure.  For a
	smooth manifold $M$, we define now the differentiable Hochschild
	cochains which vanish on constants as
	\begin{align*}
		D_\poly^k(M)
		:=
		\mathrm{DiffOps}_{n.c.}(\Cinfty(M)^{\tensor {k+1}}, \Cinfty(M))
	\end{align*}
	as a subspace of $C^k(\Cinfty(M))[1]$. Note that the Gerstenhaber
	bracket and the Hochschild differential obviously restrict to
	$D_\poly^\bullet(M)$.
\end{example}
\begin{example}[Curved Lie algebra]
	\label{ex:curvedlie}
	A basic example of an $L_\infty$-algebra is that of a (curved)
	Lie algebra $(\mathscr{L},R,\D,[\argument,\argument])$ by
	setting $Q_0^1(1)={ -}R$, $Q_1^1={ -}\D$, $Q_2^1(\gamma\vee\mu)={
		-(-1)^{|\gamma|}}[\gamma,\mu]$ and $Q_i^1=0$ for all $i\geq
	3$.  Here we denoted the degree in $\mathscr{L}[1]$ by
	$|\cdot |$.
\end{example}
Suppose $(\mathscr{L},Q)$ is an $L_\infty$-algebra. We call an element
$\gamma\in\mathscr{L}[1]$ \emph{central} if
\[Q_{n+1}^1(\gamma\vee x_1\vee\ldots\vee x_n)=0\] for all $n\geq 1$.
%
%
\begin{definition}[Curved $L_\infty$-morphism]
Let us consider two $L_\infty$-algebras $(\mathscr{L},Q)$ and
$(\widetilde{\mathscr{L}},\widetilde{Q})$.  A degree $0$ counital
coalgebra morphism
\begin{equation*}
	F\colon 
	\ucocom(\mathscr{L}[1])
	\longrightarrow 
	\ucocom(\widetilde{\mathscr{L}}[1])
\end{equation*}
such that $FQ = \widetilde{Q}F$ and $F(1)=1$ is said to be a
\emph{curved $L_\infty$-morphism}.
\end{definition} 
A coalgebra morphism $F$ from $\ucocom(\mathscr{L})$ to
$\ucocom(\widetilde{\mathscr{L}})$ such that $F(1)=1$ is uniquely determined by its
components (also called \emph{Taylor coefficients})
\begin{equation*}
	F_n^1\colon \bigvee^n(\mathscr{L}[1])\longrightarrow \widetilde{\mathscr{L}}[1],
\end{equation*}
where $n\geq 1$. Namely, we set $F(1)=1$ and use the formula
\begin{equation*}
	F(\gamma_1\vee\ldots\vee\gamma_n)=
\end{equation*}
\begin{equation}
	\label{coalgebramorphism}
	\sum_{p\geq1}\sum_{\substack{k_1,\ldots, k_p\geq1\\k_1+\ldots+k_p=n}}
	\sum_{\sigma\in \mbox{\tiny Sh($k_1$,..., $k_p$)}}\frac{\epsilon(\sigma)}{p!}
	F_{k_1}^1(\gamma_{\sigma(1)}\vee\ldots\vee\gamma_{\sigma(k_1)})\vee\ldots\vee 
	F_{k_p}^1(\gamma_{\sigma(n-k_p+1)}\vee\ldots\vee\gamma_{\sigma(n)}),
\end{equation}
where $\mathrm{Sh}(k_1,\dots,k_p)$ denotes the set of $(k_1,\ldots,
k_p)$-shuffles in $S_n$.
The element $\pi\in \mathscr{L}^1$ is called a \emph{Maurer--Cartan element} if it
satisfies the equation
	\begin{align*}
	Q_0^(1)+\sum_{k\geq 1} \frac{1}{k!}Q^1_k(\pi^{\vee k})=0,
	\end{align*}
which boils down in case the $L_\infty$-algebra is a curved Lie algebra to the equation 
\begin{equation*}
	R+\D\pi+\frac{1}{2}[\pi,\pi]
	=
	0.
\end{equation*}
Two Maurer--Cartan elements $\pi_0$ and $\pi_1$ are called homotopic, if there exists a path 
(smooth, polynomial, depending on the context) $\pi_t$ of Maurer--Cartan elements together with a path 
$\lambda_t$ of degree $0$ elements, such that $\pi_t\at{t=i}=\pi_i$ for $i\in \{0,1\}$ and 
	\begin{align*}
	\frac{\D}{\D t}\pi_t=\sum_{k\geq 0}\frac{1}{k!}Q_{k+1}^1(\pi_t^{\vee k}\vee \lambda_t),
	\end{align*}
which is case of a curved Lie algebra boils down to 
	\begin{align*}
	\frac{\D}{\D t}\pi_t=-\D \lambda_t -[\pi_t,\lambda_t].
	\end{align*}
%
For a later use we recall briefly homotopy equivalences between $L_\infty$-morphisms: let $(\mathscr{L},Q),(\mathscr{L}',Q')$ be two 
$L_\infty$-algebras and consider the space $\mathrm{Hom}(\Sym(L[1]),L')$ 
of graded linear maps, which has the structure of an $L_\infty$-algebra given by the maps
$\widehat{Q}_0$ is the map sending $1$ to $Q'_0$, 
\begin{equation}
	  \label{eq:DiffonConvLinfty}
	  \widehat{Q}^1_1 F
		=
		Q'^1_1 \circ F - (-1)^{\abs{F}} F \circ Q
	\end{equation}
	and 
	\begin{equation}
		\label{eq:BracketonConvLinfty}
		\widehat{Q}^1_n(F_1\vee \cdots \vee F_n)
		=
		(Q')^1_n \circ 
		(F_1\vee F_2\vee \cdots \vee F_n).
	\end{equation}
	Here $\abs{F}$ denotes the degree in 
	$\mathrm{Hom}(\Sym(L[1]),L')[1]$.
	The $L_\infty$-algebra $(\mathrm{Hom}(\Sym(L[1]),L'),\widehat{Q})$ is 
	called \emph{convolution $L_\infty$-algebra} and its 
	Maurer-Cartan elements which vanish on $1$ can be identified with $L_\infty$-morphisms. 
	\begin{definition}[Homotopy Equivalence]
	\begin{definitionlist}
		\item Two $L_\infty$-morphisms $F$ and $G$ are now called homotopic if they are 
	homotopic as Maurer--Cartan elements in $(\mathrm{Hom}(\Sym(L[1]),L'),\widehat{Q})$, such that all the 
	involved paths vanish on $1$.
		\item $F\colon \mathscr{L}\to \mathscr{L}'$ is a homotopy equivalence,
	if there exists a $G\colon \mathscr{L}'\to \mathscr{L}$ such that $F\circ G$ and $G\circ F$ are homotopic to the respective 
	identities. In case that the $L_\infty$-algebras are actually are non-curved, this notion coincides with the notion 
	of quasi-isomorphism.
	\end{definitionlist}
	\end{definition}
	For further details we refer the reader to \cite{kraft:2024}. 
%
%

%
%
%

\section{A zig-zag of $G_\infty$-quasi-isomorphisms}
\label{sec:zigzag}

In this section we construct an zig-zag connecting the
Hochschild complex and the multivector fields. This zig-zag 
is a refinement of a similar zig-zag constructed in 
\cite{dolgushevetal:2007} and it is basically included in there. 
For convenience, we recall its construction and state its refinement 
explicitly.  
Note that this zig-zag is a zig-zag of 
$G_\infty$-algebras, which we only briefly recall. The interested
reader can find more details in \cite{vallette:2012}.
%
\begin{definition}[dg Gerstenhaber algebra]
	A dg Gerstenhaber algebra is a triple $(\algebra A^\bullet,
	\D, m , [\argument, \argument])$, where $\algebra A^\bullet$
	is a graded space, $\D \colon \algebra A^\bullet\to \algebra
	A^\bullet$ is a differential of degree $1$, $m \colon \algebra
	A^\bullet \times \algebra A^\bullet \to \algebra A^\bullet$,
	$m(a,b) = ab$ is a graded commutative product of degree $0$
	and $[\argument, \argument] \colon \algebra A^\bullet[1]\times
	\algebra A^\bullet[1] \to \algebra A^\bullet[1]$ is a graded
	Lie bracket, satisfying the following properties:
	\begin{equation}
		\D (ab)
		=
		\D(a)b + (-1)^{\abs a} a \D(b),
	\end{equation}
	\begin{equation}
		[a, bc]
		=
		[a,b]c + (-1)^{(\abs a -1)\abs b} b[a,c]
	\end{equation}
	and
	\begin{equation}
		\D([a, b])
		=
		[\D a, b] + (-1)^{\abs a } [a,\D b]
	\end{equation}
	for any homogeneous element $a,b,c \in \algebra 
	A^\bullet$.
\end{definition}
%
%
Let $\cocom(V)$ be
the cofree conilpotent cocommutative coalgebra cogenerated by a graded
vector space $V$, which is realized by $\cocom(V)=\oplus_{k\geq 1}
\bigvee^k V$ of symmetric tensors and let $\Delta_{\sh}$ be the reduced
shuffle coproduct, see \cite{kraft:2024} for details. Furthermore, $\colie(W)$ 
denotes the cofree conilpotent Lie coalgebra
cogenerated by a graded vector space $W$. It is constructed as
follows: consider the cofree conilpotent coalgebra $(T^c(W),\Delta)$
which is the tensor coalgebra with the deconcatenation coproduct. It
can be made into a bialgebra with the product
\begin{align*}
	\mu(w_1\tensor\cdots \tensor w_l, w_{l+1}\tensor 
	\cdots \tensor w_k)= \sum_{\sigma \in \mathrm{Sh}(l,k-l)}\epsilon(\sigma)w_{\sigma^{-1}(1)}\tensor
	\cdots\tensor w_{\sigma^{-1}(k)}.
\end{align*}
One can show that the graded cocommutator $\Delta_L$
preserves
$\cc{T}^c(W):=\bigoplus_{k\geq 1} W^{\tensor k}$ and contains
\begin{align*}
	\mu(\cc{T}^c(W),\cc{T}^c(W))\subseteq \cc{T}^c(W)
\end{align*}
as a Lie coideal. We define now 
\begin{align}
	\colie(V)=\cc{T}^c(W)/\mu(\cc{T}^c(W),\cc{T}^c(W))
\end{align}
to be the quotient graded Lie coalgebra.  For a tensor
$w_1\tensor\cdots \tensor w_k\in\cc{T}^c(W)$ its equivalence class in
$\colie(V)$ by $(w_1|\cdots|w_k)$. Note that $\colie(V)$ is
canonically graded by tensor powers. For details on the cofree Lie
coalgebras and Lie coalgebras in general see \cite{michaelis}.
\begin{definition}[$G_\infty$-algebra]\label{Def:Ginftyalg}
	\begin{definitionlist}
		\item A $G_\infty$-algebra is a graded vector space
		$\algebra A^\bullet$ together with a degree $+1$
		endomorphism $q$ on $\gerst^\vee(\algebra{A}^\bullet) :=
		\cocom(\colie(\algebra{A}^\bullet[1])[1])[-2]$ such
		that $q^2=0$, 
		$$(q\tensor \id +\id\tensor q)\circ
		\Delta_\sh=\Delta_\sh\circ q
		\quad
		\text{and}
		\quad
		(q\tensor \id
		+\id\tensor q)\circ \delta=\delta\circ q.$$
		\item A $G_\infty$-morphism $F\colon
		(\algebra{A}^\bullet,q_\algebra{A})\to(\algebra{B}^\bullet,q_\algebra{B})$
		is a degree $0$ map $F\colon
		\gerst^\vee(\algebra{A}^\bullet)\to
		\gerst^\vee(\algebra{B}^\bullet)$ such that
		\begin{align*}
			\Delta_\sh\circ F 
			&=
			F\tensor F\circ \Delta_\sh,
			\\
			\delta\circ F
			&=
			F\tensor F\circ \delta,
			\\
			F\circ q_\algebra{A}
			&=q_\algebra{B}\circ F.
		\end{align*}
	\end{definitionlist}
\end{definition}
It should be noted that $\gerst^\vee$ is in fact the Koszul dual operad of the Gerstenhaber
operad and $\gerst^\vee(V)$ is the cofree Gerstenhaber coalgebra cogenerated by
$V$, see \cite{getzler} for details. Nevertheless, we are not gonna use this fact and we see 
the above equation as a definition.   

Any endomorphism $q$ on $\gerst^\vee(\algebra{A}^\bullet)$ fulfilling the
requirements from Definition \ref{Def:Ginftyalg} is uniquely
determined by its Taylor coefficients
\begin{equation}\label{Eq:TayCoeGinfty}
	q^{n_1,\dots,n_k}\colon \colie^{n_1}(\algebra{A}[1])[1]\vee\cdots \vee\colie^{n_k}(\algebra{A}[1])[1]\to \algebra{A}[3]
\end{equation}
as can be found in  \cite[Section 10.1.8]{loday.vallette:2012a}.
Given Taylor coefficients as above, their extensions to coderivations of $\gerst^\vee(\algebra{A})$
are referred as 
\emph{straight shuffle extensions}. A detailed discussion thereof can be found in 
\cite{vallette:2012}.

A dg Gerstenhaber algebra is an example of $G_\infty$-algebra. 
For our purposes it is important to show the
relation between $L_\infty$ and $G_\infty$ structures, as proved in
the following Proposition.
\begin{proposition}
	For any $G_\infty$-algebra $\algebra{A}$ there is an  induced $L_\infty$-algebra
	structures on $\algebra{A}[1]$.
	Moreover, let $F\colon \algebra{A}\to \algebra{B}$ be a
	$G_\infty$-morphism, then it induces an $L_\infty$-morphism from
	$\algebra{A}[1]$ to $\algebra{B}[1]$.
\end{proposition}
\begin{proof}
	Note that
	\begin{align*}
		\cocom(\algebra{A}[2])\hookrightarrow 
		\cocom(\colie(\algebra{A}[1])[1])
	\end{align*}
	is an inclusion of coalgebras. Using that $(q\tensor \id
	+\id\tensor q)\circ \delta=\delta\circ q$ for the
	$G_\infty$-structure on $\algebra{A}$, one can show that $q$
	restricts to $\cocom(\algebra{A}[2])$ and thus gives us an
	$L_\infty$-structure on $\algebra{A}[1]$.  In terms of Taylor
	coefficients like in Equation \eqref{Eq:TayCoeGinfty}, the
	Taylor coefficients of the $L_\infty$ are given by
	\begin{align*}
		q^{1,\dots,1}\colon \algebra{A}[2]\vee\cdots\vee\algebra{A}[2]\to \algebra{A}[2][1]
	\end{align*}
	With the same
	reasoning, one can show that the $G_\infty$-morphism $F$,
	restricted to $\cocom(\algebra{A}[2])$ takes only values in
	$\cocom(\algebra{B}[2])$ and thus is an $L_\infty$-morphism of
	the corresponding $L_\infty$-algebras.
\end{proof}
For any associative algebra $(A,\mu_0)$ the Hochschild complex does
not only carry the structure of a DGLA. In fact there is the so-called
cup product which is defined for two cochains $D$ and $E$ by
\begin{equation}
	D \cup E (a_1,\dots,a_{d+e})
	=
	D(a_1,\dots, a_d)E(a_{d+1},\dots,a_{d+e}),
\end{equation}
where $a_1,\dots,a_{d+e}\in A$. This product turns
$(C^{\bullet}(A),\partial, \cup)$ into a graded differential
associative algebra.  Note that the cup product is not graded
commutative, so it is not part of a Gerstenhaber
structure, but one can prove that the Hochschild complex $C^\bullet(A)$ of a
commutative algebra $A$ has a natural $G_\infty$-structure 
%
with Taylor components $q^{n_1,\cdots,n_k}$
satisfying the following properties
\begin{align}
	\label{item: Liebial}
	q^{n_1,\cdots,n_k}=0
\end{align}
for all $k\geq 3$,
and
\begin{align}
	\label{eq:good}
	q^{1,n}(X,C)=0
\end{align}
for any $X\in \Der(A)$ and arbitrary 
$C\in \colie^{n}(C^{\bullet}(A)[1])$ 
for $n\geq 2$, see \cite[Theorem 2.1]{tamarkin:1999}.
Moreover, for $D_1,D_2\in C^\bullet(A)$ we have 
	\begin{align*}
		q^{2}(D_1,D_2)=\frac{1}{2}(D_1\cup D_2+
		(-1)^{|D_1||D_2|}D_2\cup D_1). 
	\end{align*}
%
%
%
%
%
%
\begin{lemma}
	\label{lemma:liebialgcolie}
	Let $\algebra{A}^\bullet$ be a $G_\infty$-algebra satisfying \eqref{item: Liebial}. Then 
	\begin{align}
		\mathscr{L}_\algebra{A}
		:= 
		(\colie (\algebra A [1]), [\argument, \argument], \D , \delta^c)
	\end{align}
	is a dg Lie bialgebra with bracket 
	\begin{align}
		[C,D]
		= 
		[C,D]^{n_1,n_2}   
	\end{align}
	for $C\in \colie^{n_1}(\algebra A [1])$ and
	$D\in\colie^{n_2}(\algebra A [1])$, where
	$[\argument,\argument]^{n_1,n_2}$ is the straight shuffle
	extension of $q^{n_1,n_2}$.  The differential $\D$ given by
	the straight shuffle extension of the $q^{n}$'s and $\delta^c$
	is the cofree Lie cobracket.
\end{lemma}
\begin{proof}
In \cite{hinich}, $G_\infty$-algebras fulfilling property  \eqref{item: Liebial} are called 
$\widetilde{\mathscr{B}}$-algebras. Inside this note there is also a detailed proof of the above lemma. 
\end{proof}

\begin{remark}
	Note that any Gerstenhaber algebra (in particular
	$\multivect^\bullet(A)$), satisfies
	property~\eqref{item: Liebial}. Thus, it can be turned into a dg
	Lie bialgebra via Lemma~\ref{lemma:liebialgcolie}.
\end{remark}
Consider
\begin{align}
	L_A
	:=
	\bigoplus_{n=1}^\infty
	\big\{
	\phi\in\colie^n(\multivect^{\bullet}(A)[1])
	\vert \ 
	\deg(\phi)\in\{-n,-n+1\} \big\} \subset 
	\mathscr{L}_{\multivect^\bullet(A)}.  
\end{align}
One can easily check 
that $L_A$ is actually a sub-Lie
bialgebra of $\mathscr{L}_{\multivect^\bullet(A)}$ and property~\eqref{eq:good} 
ensures that it is also a sub-Lie bialgebra of
$\mathscr{L}_{C^\bullet(A)}$ (a detailed proof of this claim can be found in
\cite{dolgushevetal:2007}). One can actually see that
\begin{align}
	(a_1|\cdots|a_k) \in L_A \subseteq \mathscr{L}_{\multivect^\bullet(A)}
\end{align}
if all besides at most one of the $a_i$'s are elements of $A$ and the
remaining one is a derivation.  \newline
Now we apply a general observation
about dg Lie bialgebras:
\begin{lemma}\label{Lem: dgLb to Gerst}
	\label{lem:dggerst}
	Let $(L,\D,[\argument,\argument],\delta)$ be a dg Lie
	bialgebra, then $\comm(L[-1])$ can be endowed with the
	structure of a dg Gerstenhaber algebra. In
	particular, the bracket is given by
	\begin{align}
		\{\ell_1\vee\cdots\vee \ell_k,m_1\vee\cdots\vee m_n\}
		=
		\sum_{i,j}
		\epsilon(\sigma_{i,j})
		[\ell_i,m_j]\vee \ell_1\vee\cdots\vee \ell_{k} \vee m_1
		\vee\cdots \vee m_n,
	\end{align}
	the product is the free commutative product and the
	differential is given by the extension of $\D$ and $\delta$ as
	derivations of the product. Furthermore, a morphism of dg Lie
	bialgebras $f\colon L\to L'$ induces a morphism of dg
	Gerstenhaber algebras $\comm(f)\colon
	\comm(L[-1])\to\comm(L'[-1])$.
\end{lemma}
The next proposition is in some sense the dual version of Lemma \ref{Lem: dgLb to Gerst}
\begin{proposition}
	\label{prop:dgliebi}
	Let $(L, \D_L, [\argument, \argument]_L, \delta_L)$ be a dg Lie bialgebra. 
	\begin{propositionlist}
		\item $\flie (\cocom (L [1])[-1])$ is a Lie bialgebra
		in a natural way, such that the canonical map
		\begin{align*}
			\flie(\cocom(L[1])[-1])\longrightarrow L
		\end{align*}
		is a morphism of dg Lie bialgebras;
		\item Let $\algebra A^\bullet$ be a $G_\infty$-algebra
		satisfying condition \eqref{item: Liebial}. Then the
		dg Gerstenhaber algebra \\ $\comm (\flie (\cocom
		(\colie (\algebra A[1]))[1])[-1])[-1])$ coincides
		with $\gerst((\gerst^\vee (\algebra A))$.
	\end{propositionlist}
\end{proposition}
\begin{proof}
	\begin{propositionlist}
		\item The differential $Q\colon \cocom(L[1])\to
		\cocom(L[1])[1]$ comes from the differential $\D_L$
		and the Lie bracket $[\argument, \argument]_L$ in a
		natural way.  The cobracket comes from $\delta_L$ in
		the following way: for $\ell\in L$ we use $\delta_L$
		and extend it to $\cocom(L[1])$ by the coLeibniz
		identity
		\begin{equation}
			\delta_\cocom(a\vee b) =\delta_\cocom(a)\Delta_\sh (b)+
			(-1)^{|a|}\Delta_\sh (a)\delta_\cocom(b)
		\end{equation}
		for $a,b\in \cocom(L[1])$, which is a bi-coderivation
		of the cocommutative coproduct on $\cocom(L[1])$ of
		degree $-1$ turning $\cocom(L[1])$ into a
		Gerstenhaber coalgebra. Now we consider the free Lie
		algebra of $\cocom(L[1])[-1]$ and define the
		differential to be the extension of $Q$ and
		$\Delta_\sh$ as (graded, degree 1) derivations of the
		free Lie bracket.  The Lie cobracket $\delta\colon
		\flie(\cocom(L[1])[-1])\to
		\Anti^2\flie(\cocom(L[1])[-1])$ is the extension of
		the co Gerstenhaber bracket by the rule
		\begin{equation}
			\delta([x,y])
			=
			[\delta(a),b\tensor 1 + 1\tensor b]
			+[a\tensor 1+ 1\tensor a, \delta(b)],
		\end{equation}
		to the free Lie algebra, which gives us automatically
		the compatibility with the Lie bracket.
		
		Finally we need to check that the canonical dg Lie algebra 
		morphism 
		\begin{align*}
			P\colon \flie(\cocom(L[1])[-1])\to L
		\end{align*}
		is compatible with the cobracket.  Recall that the
		morphism $P$ is given by the extension of the
		projection $\cocom^1(L[1])[-1]=L$ as a Lie algebra
		morphism. Note that this immediately implies that 
		\begin{align*}
		P([x_1,[\cdots[x_{k-1},x_k]\cdots])=0
		\end{align*}
		if there exists $j\in \{1,\dots,k\}$, such that 
		 $x_j\in \cocom^{\geq 2}(L[1])$. 
		For the  Gerstenhaber cobracket $\delta_\cocom$, we have 
		by construction 
		$\delta_{\cocom}(\cocom^k(L[1]))\subseteq 
		\bigoplus_{i+j=k-1}\cocom^{i+1}(L[1])\vee
		\cocom^{j+1}(L[1])$ and thus
		\begin{align}
			P\tensor P(\delta([x_1,[\cdots[x_{k-1},x_k]\cdots]))
			=
			0    
		\end{align}
		if there exists $j\in \{1,\dots,k\}$, such that 
		 $x_j\in \cocom^{\geq 2}(L[1])$.  In this case we
		also get by the above discussion
		$\delta_L(P([x_1,[\cdots[x_{k-1},x_k]\cdots])=0$ by
		the definition of $P$. The rest of the statement
		follows automatically.
		\item This is an immediate consequence of the above item,
		since the extension of the cofree Lie cobracket of 
		$\colie(\algebra{A}[1])$ gives the cofree Gerstenhaber cobracket
		on $\cocom(\colie(\algebra{A}[1])[1])$ and thus the
		Gerstenhaber coalgebra on 
		$\cocom(\colie(\algebra{A}[1])[1])[-2]$ equals the one on $\gerst^\vee(\algebra{A}[1])$. 
	\end{propositionlist}
\end{proof}
Let us apply Lemma~\ref{lem:dggerst} to the dg Lie bialgebras
$\mathscr{L}_{C^\bullet(A)}$, $\mathscr{L}_{\multivect^\bullet(A)}$
and $L_A$ to obtain the zig zag
\begin{center}
	\begin{tikzcd}
		\comm(\mathscr{L}_{C^\bullet(A)}[-1]) &
		\comm(L_A[-1])\arrow[l]\arrow[r]&
		\comm(\mathscr{L}_{\multivect^\bullet(A)}[-1])
	\end{tikzcd}
\end{center}
Recall that in \cite{dolgushevetal:2007} a $G_\infty$-zig zag
has been constructed connecting $\multivect^\bullet(A)$ and
$C^\bullet(A)$, via the following diagram
\begin{center}
	\begin{tikzcd}
		C^\bullet(A)\arrow[r, "\chi^{C^\bullet(A)}"] &
		\gerst(\gerst^\vee(C^\bullet(A))) \\
		& \comm ((\flie\cocom (L_A[1])[-1])[-1]) \arrow[u]\arrow[d]\\
		\multivect^\bullet(A)
		& \gerst(\gerst^\vee(\multivect^\bullet(A)))\arrow[l]\\
	\end{tikzcd}
\end{center} 
Here all the $G_\infty$-algebras are actually Gerstenhaber
algebras except $C^\bullet(A)$, while all the arrows are
Gerstenhaber morphism besides $\chi^{C^\bullet(A)}$ which is
the unit of the $G_\infty$-$\gerst^\vee$-adjunction.  Using
Proposition~\ref{prop:dgliebi} and Lemma~\ref{lem:dggerst}, we
obtain
\begin{center}
	\begin{tikzcd}
		C^\bullet(A)\arrow[r, "\chi^{C^\bullet(A)}"] &
		\gerst(\gerst^\vee(C^\bullet(A))) \arrow[r]&
		\comm(\mathscr{L}_{C^\bullet(A)}[-1])\\
		& \comm (\flie(\cocom (L_A[1])[-1])[-1]) \arrow[u]\arrow[d]
		\arrow[r]
		& \comm (L_A[-1]) \arrow[u]\arrow[d]\\
		\multivect^\bullet(A)
		& \gerst(\gerst^\vee(\multivect^\bullet(A)))\arrow[l]
		\arrow[r]
		& \comm(\mathscr{L}_{\multivect^\bullet(A)}[-1])\\
	\end{tikzcd}
\end{center} 
The lowest row can actually be replaced by the morphism
$\comm(\mathscr{L}_{\multivect^\bullet(A)}[-1])
\longrightarrow \multivect^\bullet(A)$, which is given by the
projection $\mathscr{L}_{\multivect^\bullet(A)}[-1] =
\colie(\multivect^\bullet(A)[1])[-1] \longrightarrow
\colie^1(\multivect^\bullet(A)[1])[-1] =
\multivect^\bullet(A)$ extended as a morphism of commutative
algebras to
$\comm(\mathscr{L}_{\multivect^\bullet(A)}[-1])\longrightarrow\multivect^\bullet(A)$. It
turns out that this map is a morphism of Gerstenhaber
algebras.  
Summarizing, we obtain the following theorem, which is basically already contained in \cite{dolgushevetal:2007}.
\begin{theorem}
	\label{thm:zigzag}
	Let $A$ be a commutative algebra over a field of
	characteristic $0$, then there is a natural diagram in the
	category of $G_\infty$-algebras
	\begin{center}
		\begin{tikzcd}
			\comm(\mathscr{L}_{C^\bullet(A)}[-1]) & 
			\comm (L_A[-1]) \arrow[l]\arrow[r]&
			\comm(\mathscr{L}_{\multivect(A)}[-1])\arrow[d]\\
			C^\bullet(A)\arrow[u] & & 
			\multivect^\bullet(A)
		\end{tikzcd}
	\end{center} 
\end{theorem}
\begin{remark}
Note that the zig-zag from Theorem \ref{thm:zigzag} exists for every commutative algebra, but we need extra assumptions to see that all the morphisms are actually quasi-isomorphisms.  
\end{remark}
%

%
%
%
\section{A $\lie{g}$-adapted quasi-inverse} 
\label{sec:gadapted}

The first step of this section consists of introducing the complexes
which contain the data of Hamiltonian actions in the classical and
quantum setting.  Those complexes turn out to have a \emph{curvature}
and, as a second step, we discuss the compatibility of formality maps
with such curvatures.

\subsection{Equivariant Complexes as curved Lie algebras}
Let us consider a Lie group action $\Phi\colon G\times M\to M$. In the
following we use a superscript $G$ to refer to $G$-invariant elements.
\begin{definition}[Equivariant Multivectors]
	\label{def:equiv1}
	The equivariant multivector DGLA is given by the complex
	$ T^\bullet_\lie{g} (M)$ defined by
	\begin{align*}
		T_\lie{g}^k (M)
		=
		\bigoplus_{2i+j=k} (\Sym^i\lie{g}^* \tensor \Secinfty(\Anti^{j+1}TM) )^G
		=
		\bigoplus_{2i+j=k} (\Sym^i\lie{g}^* \tensor T_{\poly}^j(M) )^G,
	\end{align*}
	together with the trivial differential and the following Lie bracket
	\begin{align*}
		[\alpha\tensor X, \beta\tensor Y]_{\lie{g}}
		=
		\alpha\vee \beta\tensor[]   [X,Y]_{\mathrm{SN}} 
	\end{align*}
	for any $\alpha\tensor X, \beta\tensor Y \in  T_\lie{g}^\bullet(M)$.
\end{definition}
Here $[\argument,\argument]_{\mathrm{SN}}$ refers to the usual
Schouten--Nijenhuis bracket on $T_{\poly}(M)$.  Notice that invariance
with respect to the group action means invariance under the
transformations $\Ad_g^*\tensor \Phi_g^*$ for all $g\in G$.
\begin{remark}
	We can interpret this complex in terms of polynomial maps
	$\lie{g}\to T_{\poly}^j(M)$ which are equivariant with respect
	to adjoint and push-forward action.
	Furthermore, we introduce the canonical linear map
	\begin{align}
		\lambda
		\colon 
		\lie{g} \ni \xi \longmapsto \xi_M\in T^0_{\poly}(M),
	\end{align}	 
	where $\xi_M$ denotes the fundamental vector field
	corresponding to the action $\Phi$.  It turns that $\lambda $
	is central and as a consequence we can turn $
	T^\bullet_\lie{g} M$ into a \emph{curved} Lie algebra with
	curvature $\lambda$.
\end{remark}
Now let $(M, \pi)$ be a Poisson manifold. Recall that a momentum map
for the action $\Phi$ is a map $J\colon \lie{g} \to \Cinfty(M)$ such
that
\begin{equation}
	\label{eq:momentummap}
	\xi_M = \{ \argument, J_\xi \}
	\quad
	\text{and}
	\quad
	J_{[\xi, \eta]} 
	=
	\{ J_\xi,  J_\eta \},
\end{equation}
where $\{\argument, \argument\}$ denote the Poisson bracket
corresponding to $\pi$. An action $\Phi$ admitting a momentum map is
called \emph{Hamiltonian}.  In the following we recall a
characterization of Hamiltonian actions in terms of equivariant
multivectors, see \cite{esposito:2022}.
\begin{lemma}
	\label{lem:MChamiltonian}
	The curved Maurer--Cartan elements 
	\begin{align*}
		\mathrm{MC}_\lambda(T_\lie{g}^\bullet (M))=\{ \Pi\in T_\lie{g}^1 (M)\ \mid \ \lambda +\frac{1}{2}[\Pi,\Pi]_\lie{g}=0\}
	\end{align*}
	are equivalent to pairs $(\pi,J)$, where $\pi$ is a 
	$G$-invariant Poisson structure and $J$ is a momentum
	map $J\colon \lie{g} \to T^{-1}_{\poly}(M)$.
\end{lemma}
In a similar fashion, we recall the definition of equivariant
polydifferential operators. 
\begin{definition}[Equivariant Hochschild cochains]
	\label{def:equiv2}
	The DGLA of equivariant multidifferential operators is given by
        the complex $D^\bullet_{\lie{g}}(M)$ defined as
	\begin{align*}
		D^k_{\lie{g}}(M)
		= 
		\bigoplus_{2i+j=k} (\Sym^i\lie{g}^* \tensor D_{\poly}^j(M))^G.
	\end{align*}
	The DGLA structure maps are induced by interpreting
	$D^\bullet_{\lie{g}}(M)$ as equivariant polynomial maps \\$\lie{g} \to
	D_{\poly}(M)$. More precisely, we have
	\begin{align*}
		(\partial^\lie{g} D)(\xi)
		=
		\partial(D(\xi))
		=
		[\mu,D(\xi)]
	\end{align*}
	and
	\begin{align*}
		[D_1,D_2]_\lie{g}(\xi)
		=
		[D_1(\xi),D_2(\xi)],
	\end{align*} 
	where we denote by $\partial $ and $[\argument,\argument]$ the
	usual structure maps of the multidifferential operators DGLA
	and $\mu$ the pointwise multiplication of $\Cinfty(M)$.
\end{definition}
Let $M$ be now equipped with a
$G$-invariant star product $\star$.  Recall that a map $H\colon
\lie{g} \to \Cinfty(M)[[\hbar]]$ is called a \emph{quantum momentum map}
if
\begin{align*}
	\Lie_{\xi_M}
	=
	-\frac{1}{\hbar}[H(\xi),-]_\star 
	\ \text{ and }\ 
	\frac{1}{\hbar}[H(\xi),H({\eta})]_\star
	=
	H([\xi,\eta]).
\end{align*}
%
As in the equivariant multivector field case, we interpret quantum
momentum maps as (curved) Maurer--Cartan elements. With a slight abuse
of notation, we define
\begin{align*}
	\lambda
	\colon
	\lie{g} \ni \xi \longmapsto \Lie_{\xi_M}\in D^0_{\poly}(M),
\end{align*}	
and see that $\lambda \in C^2_{\lie{g}}(M)$ is central and 
moreover $\partial^\lie{g} \lambda = 0$. This means we can 
see $C^\bullet_\lie{g}(M)$ as a curved Lie algebra with 
the above structures and curvature $\lambda$. 
The proof of the following Lemma can be found in \cite{esposito:2023}.
\begin{lemma}\label{Lem: MCtoStar} 
	A curved formal Maurer-Cartan element $\Pi\in \hbar
	C^1_{\lie{g}}(M)[[\hbar]]$, i.e. an element $\Pi$ satisfying
	\begin{align}
		\label{eq:MCquantum}
		\hbar^2\lambda + \partial^\lie{g} \Pi + \frac{1}{2}[\Pi,\Pi]_\lie{g}
		=
		0,
	\end{align}
	is equivalent to a pair $(m_\star,H)$, where $m_\star$ defines
	a formal star product via $\mu+\hbar m_\star$ with a quantum
	momentum map $H\colon \lie{g} \to \Cinfty(M)[[\hbar]]$.
\end{lemma}
Such a pair consisting of an invariant star product $\star = \mu +
\hbar m_\star$ and a quantum momentum map $H$ is also called
\emph{equivariant star product}.  Note that we always assume
the filtration by powers of $\hbar$.

\subsection{Compatibility with Curvatures}

The main goal of this paper consists in constructing a one-to-one
correspondence between equivalence classes of $G$-invariant star
product with quantum momentum maps and equivalence classes of
$G$-invariant formal Poisson structures with formal momentum maps. In
other words, we aim at constructing an $L_\infty$-morphism
\begin{equation*}
	F^\lie{g} \colon
	T_{\mathfrak{g}}(M) 
	\longrightarrow 
	D_{\mathfrak{g}}(M)
\end{equation*}
of the \emph{curved} Lie algebras defined above which is also an
homotopy equivalence.  The idea is to choose a $G$-invariant formality
map $F\colon T_\poly(M)\to D_\poly(M)$, which is compatible with
curvatures i.e. $F_1(\xi_M) = \Lie_{\xi_M}$ for all $\xi\in \lie{g}$,
and extend it to an $L_\infty$-morphism $F^\lie{g}\colon
T_\lie{g}(M)\to D_\lie{g}(M)$ by setting:
\begin{align}
	\label{Eq: ExtForm}
	F^\lie{g}_k(P_1\tensor X_1\vee\dots\vee P_k\tensor X_k)=P_1\vee\dots\vee P_k\tensor F_k(X_1\vee\dots \vee X_k).
\end{align}	   
The map $F^\lie{g}$ defined above is obviously compatible with
corresponding brackets and differentials.  Nevertheless, the
compatibility with the curvature is more delicate.  First, let us
start with the following general observation.
\begin{proposition}
	\label{prop:CompatibilityMorphismCurvatur}
	Let $F$ be an $L_\infty$-morphism between 
	$L_\infty$-algebras $(\mathscr{L},Q)$ and
	$(\widetilde{\mathscr{L}},\widetilde{Q})$. In addition, let
	$Q_0 \in \mathscr{L}[1]^1, \widetilde{Q}_0 \in
	\widetilde{\mathscr{L}}[1]^1$ be central with
	respect to the $L_\infty$-structures $Q$ and $\widetilde{Q}$,
	resp. Then, the following holds:
	\begin{propositionlist}
	    \item The coderivation $Q_0\colon \ucocom(\mathscr{L}[1])\to \ucocom(\mathscr{L}[1])$ defined by 
	    	\begin{align*}
	    	Q'(x_1\vee\cdots\vee x_k):= Q(x_1\vee \cdots \vee x_k)+ Q_0\vee x_1\vee \cdots\vee x_k
	    	\end{align*}
	    gives $\mathscr{L}$ the structure of a curved $L_\infty$-algebra with $Q'(1)=Q_0$.
		\item The structure maps of $F$
		induces a morphism of curved $L_\infty$-algebras if and only
		if
		\begin{equation}\label{Eq:flattocurved}
			F_1(Q_0) 
			=
			\widetilde{Q}_0 
			\quad \quad \text{ and } \quad \quad 
			F_k(Q_0\vee \argument)
			=
			0
			\quad \quad \forall \; k>1.
		\end{equation}
		\item Given two $L_\infty$-morphisms $F_0,F_1\colon (\mathscr{L},Q)\to (\widetilde{\mathscr{L}},\widetilde{Q})$ which are 
		homotopic via the equivalence $(F_t,\lambda_t)$, such that $F_t$ fulfills the conditions 
		\eqref{Eq:flattocurved} for all $t$ and 
		\begin{align}
			(\lambda_t)_k(Q_0\vee\argument)=0
		\end{align}
		for all $k\geq 1$, then the corresponding curved $L_\infty$-morphisms are also homotopic. 
	\end{propositionlist}
\end{proposition}
\begin{proof} 
	Let us denote the induced curved $L_\infty$-structures $Q'$ on
	$\mathscr{L}$ and $\widetilde{Q}'$ on
	$\widetilde{\mathscr{L}}$, respectively.
	\begin{propositionlist}
        \item The first part follows, since for $x_1,\dots,x_k\in \mathscr{L}[1]$, we have 
        	\begin{align*}
        	Q'\circ Q'(x_1\vee \cdots \vee x_k)&= Q^2(x_1\vee \cdots x_k) + Q(Q_0\vee x_1\vee \cdots\vee x_k)\\&
        	+Q_0\vee Q(x_1\vee \cdots \vee x_k)
	    	+ Q_0\vee Q_0\vee x_1\vee \cdots \vee x_k.
			\end{align*}        			
		The first and the last term vanish canonically and $Q(Q_0\vee x_1\vee \cdots\vee x_k)=
		-Q_0 \vee Q(x_1\vee \cdots \vee x_k)$, since $Q_0$ is closed, central and of degree $+1$.
		\item 	One only has to check if $F$ is compatible with the
		codifferentials $Q'$ and $\widetilde{Q}'$, where one gets for
		$v_1\vee \cdots \vee v_n$ with $ v_i \in \mathscr{L}[1]^{k_i}$
		and $n \neq 0$
		\begin{align*}
			\pr_{\widetilde{\mathscr{L}}[1]} \circ F \circ Q' (v_1 \vee \cdots \vee v_n)
			& =
			\pr_{\widetilde{\mathscr{L}}[1]} \circ F ( Q_0 \vee v_1 \vee \cdots \vee v_n) 
			+ Q(v_1\vee \cdots \vee v_n))  \\
			& =
			F_{n+1}^1( Q_0 \vee v_1 \vee \cdots \vee v_n)  
			+ \pr_{\widetilde{\mathscr{L}}[1]} \circ \widetilde{Q}\circ 
			F (v_1\vee \cdots \vee v_n) \\
			& \stackrel{!}{=}
			\pr_{\widetilde{\mathscr{L}}[1]} \circ \widetilde{Q}' \circ 
			F(v_1\vee \cdots \vee v_n).
		\end{align*}
		In addition, one has
		\begin{equation*}
			\widetilde{Q}_0
			=
			\widetilde{Q}' \circ F(1)
			\stackrel{!}{=}
			F \circ Q'(1)
			=
			F_1(Q_0),
		\end{equation*}
		which directly yields the above identities.
		\item 	Let now $(F_t,\lambda_t)$ be a homotopy equivalence. In other words, we have that
		$F_t\in \mathrm{Hom}(\cocom\mathscr{L}[1],\mathscr{L}')$ is a path of Maurer-Cartan elements, such that 
		\begin{align*}
			\frac{\D}{\D t} F_t=q^{F_t}(\lambda_t),
		\end{align*}
		where $q$ is the $L_\infty$-structure of the
                convolution $L_\infty$-structure induced by $Q$ and
                $\tilde{Q}$, see \cite{kraft:2024} for details.
		%
		Note that $F_t$ induces by assumption a path of curved
                $L_\infty$-morphisms, so it is a path of curved
                Maurer-Cartan elements in the curved convolution
                $L_\infty$-algebra
                $(\mathrm{Hom}(\cocom\mathscr{L}[1],\mathscr{L}'),\tilde{q})$
                induced by $\widetilde{Q}$ and $\widetilde{Q'}$. Now
                we only need to check that
		\begin{align*}
			\frac{\D}{\D t} F_t=\tilde{q}^{F_t}(\lambda_t).
		\end{align*}
		By definition $q$ and $\tilde{q}$ only differ by 
		$(q(F)-\tilde{q}(F))(x)=\pm F (Q_0\vee x)$, 
		which vanishes on $\lambda_t$ by assumption and the claim is proven. 
	\end{propositionlist}
\end{proof}
We can immediately conclude that the extended formality from
Eq.~\eqref{Eq: ExtForm} is compatible with curvature if and only if
\begin{align*}
	F_{k+1}^1(\lambda\vee\argument)=0 
\end{align*}
for any $k\geq 1$, which leads us to the following characterization.
\begin{corollary}
	Any $G$-invariant $L_\infty$-morphism $F\colon T_\poly(M)\to
	D_\poly(M)$ with $F_1(\xi_M)=\Lie_{\xi_M}$ for all $\xi\in
	\lie{g}$ can be lifted to a curved $L_\infty$-morphism
	$F^\lie{g}\colon T_\lie{g}(M)\to D_\lie{g}(M)$ if and only if
	\begin{align*}
		F_k^1(\xi_M\vee\argument)=0
	\end{align*}
	for all $\xi\in \lie{g}$. 
\end{corollary}

\subsection{$\lie g$-adapted $L_\infty$-algebras}

From the above discussion, it is now evident that the compatibility
between $L_\infty$ morphisms and the curvatures coming from the action
is crucial. 
Given its importance, we turn it into a definition and introduce the
corresponding category.
\begin{definition}[$\lie g$-adapted $L_\infty$-algebras]
	Let $\lie{g}$ be a Lie algebra. 
	\begin{definitionlist}
		\item A $\lie{g}$-adapted $L_\infty$-algebra is a pair
		$(\algebra L^\bullet, \phi)$ consisting of a
		$L_\infty$-algebra $(\algebra{L}^\bullet,
		Q_\algebra{L})$ and a linear map $\phi\colon
		\lie{g}\to \algebra{L}^0$ such that
		\begin{equation}
			Q_{\algebra{L},k+1}^1
			(\phi(\xi),x_1,\dots,x_k)
			=
			0
		\end{equation}
		for all $\xi
		\in \lie g$, $x_i\in\algebra{L}[1]$, $k \neq 1$ and 
		\begin{align}
			Q_{\algebra{L},2}^1(\phi(\xi), \phi(\eta))
			=
			\phi([\xi,\eta])
		\end{align}
		for $\xi, \eta \in \lie g$.
		\item A morphism of $\lie{g}$-adapted
		$L_\infty$-algebras $(\algebra L^\bullet, \phi)$ and
		$(\algebra H^\bullet, \psi)$ is a
		$L_\infty$-morphism $F\colon \algebra L^\bullet \to
		\algebra H^\bullet$ such that
		\begin{align}
			F^1_{k+1}
			(\phi(\xi),x_1,\dots,x_k)
			=
			0
		\end{align}
		for $k\geq 1$ and such that
		\begin{align}
			F^1_1(\phi(\xi))=\psi(\xi)
		\end{align}
		or all $\xi \in \lie g$, $x_i\in\mathscr{L}[1]$.
	\item Two morphisms $F_0$ and $F_1$ of $\lie{g}$-adapted
		$L_\infty$-algebras $(\algebra L^\bullet, \phi)$ and
		$(\algebra H^\bullet,\psi)$ are called homotopic, if there exists a 
		curve of $\lie{g}$-adapted $L_\infty$-morphisms $F_t$ and a curve of degree
		$0$ elements in $\lambda_t\in \mathrm{Hom}(\Sym\mathscr{L}[1],\mathscr{L}')$, such that
		$(F_t,\lambda_t)$ is a homotopy between $F_0$ and $F_1$ as $L_\infty$-morphisms and 
		\begin{align*}
		\lambda_t(\phi(\xi)\vee\argument)=0 \text{ for all } \xi\in \lie{g}. 
		\end{align*}
	\item A $\lie{g}$-adapted morphism $F\colon (\algebra L^\bullet, \phi)\to(\algebra H^\bullet,\psi)$
	     is called homotopy equivalence, if there exists a $\lie{g}$-adapted morphism 
	     $G\colon (\algebra H^\bullet,\psi\to)(\algebra L^\bullet, \phi)$ such that $F\circ G$ and $G\circ F$ are 
	     homotopic to the respective identities. $G$ is called $\lie{g}$-adapted quasi-inverse of $F$.  
	\end{definitionlist}
	
	\end{definition}
Note that it is clear that given two $\lie{g}$-adapted morphisms 
$F\colon \algebra L^\bullet \to \algebra H^\bullet$ and 
$G\colon \algebra {H}^\bullet \to \algebra{K}^\bullet$ that their composition
 $G\circ F\colon \algebra{L}^\bullet\to \algebra{K}^\bullet$ is again a $\lie{g}$-adapted $L_\infty$-morphism, so 
$\lie{g}$-adapted $L_\infty$-algebras form indeed a category. 

\begin{remark}
It is well known that, 
for general $L_\infty$-algebras the notions of homotopy equivalence and 
quasi-isomorphism are equivalent (see e.g. \cite{kraft:2024}). 
This is in general not the case anymore for $\lie{g}$-adapted $L_\infty$-morphisms.
Here the notion of homotopy equivalence is a strictly stronger condition.  
\end{remark}

\begin{proposition}
	Let $(\algebra L^\bullet, \phi)$ and $
	(\algebra{H}^\bullet,\psi)$ be $\lie{g}$-adapted
	$L_\infty$-algebras and let $F\colon \algebra{L}^\bullet\to
	\algebra{H}^\bullet$ be a $\lie
	g$-adapted morphism. Then
	\begin{propositionlist}
		\item the map 
		\begin{align}
			\lie{g}\ni \xi \longmapsto Q_{\algebra
				L,2}^{1}(\phi(\xi),\argument)\in \End(\algebra{L})
		\end{align}
		defines a Lie algebra action $\acts$ of $\lie{g}$ on
		$\algebra{L}$.
		\item the Taylor coefficients of $Q_{\algebra{L}}$,
		$Q_{\algebra{H}}$ and $F$ are equivariant with
		respect to this action.
	\end{propositionlist} 
\end{proposition}
\begin{proof}
	Let us set $\xi \acts x = Q_\algebra{L}(\phi(\xi)\vee x)$, for any
	$\xi \in \lie{g}$ and $x\in \algebra{L}[1]$.
	Note that the map $\phi(\xi)\vee \colon
	\Sym(\algebra{L}[1])\to \Sym(\algebra{L}[1])$ defined by
	\begin{align*}
		\phi(\xi)\vee(x_1\vee\cdots\vee x_k)=\phi(\xi)\vee x_1\vee\cdots\vee x_k
	\end{align*}
	is a coderivation of degree $-1$, which means that
	$[Q_\algebra{L},\phi(\xi)\vee]$ is a coderivation of degree
	$0$.  One can check that
	\begin{align*}
		[Q_\algebra{L},\phi(\xi)\vee](x_1\vee\cdots\vee x_k)=
		\sum_i x_1\vee\cdots (\xi\acts x_i)\vee\cdots\vee x_k.
	\end{align*}
	Thus, we obtain: 
	\begin{propositionlist}
		\item Let $\xi,\eta\in \lie{g}$ and $x\in \algebra{L}[1]$, then 
		\begin{align*}
			0&=(Q_\algebra{L}Q_\algebra{L})^{1}_3(\phi(\xi)\vee\phi(\eta)\vee x)\\&
			=[\xi,\eta]\acts x- \xi\acts(\eta\acts x)+\eta\acts(\xi\acts x)
		\end{align*}
		This means in particular that $\acts$ is a Lie algebra action. 
		\item The proofs for the equivariance are essentially
		the same for all the different Taylor coefficients,
		let us therefore prove it only for $Q_\algebra{L}$. We
		have for $\xi\in\lie{g}$ and $x_i\in \algebra{L}[1]$
		the following equation
		\begin{align*}
			Q_{\algebra{L},k}^1(\xi\acts(x_1\vee\cdots\vee x_k))
			&=
			Q_{\algebra{L},k}^1( [Q_\algebra{L},\phi(\xi)\vee](x_1\vee\cdots\vee x_k)\\&
			=
			Q_{\algebra{L},i+1}^1(\phi(\xi)\vee Q^i_{\algebra{L},k}(x_1\vee\cdots\vee x_k))\\&
			=
			Q_{\algebra{L},2}^1(\phi(\xi)\vee Q^1_{\algebra{L},k}(x_1\vee\cdots\vee x_k))\\&
			=
			\xi\acts Q_{\algebra{L},k}^1(x_1\vee\cdots\vee x_k).
		\end{align*}
	\end{propositionlist}
\end{proof}
We are ready to prove one of our main results.
\begin{theorem}\label{Thm:gQuis}
	Given two $\lie g$-adapted $L_\infty$-algebras $(\algebra
	L^\bullet,Q_{\algebra L}, \phi)$, $(\algebra H^\bullet,Q_\algebra{H}, \psi)$ and let
	$F\colon \algebra L^\bullet \to \algebra H^\bullet$ be a $\lie
	g$-adapted morphism such that the first component of $F$ is
	part of a chain homotopy equivalence
	\begin{equation}
		\label{eq:homotopyret}
		\begin{tikzcd} 
			\arrow[loop left, distance=3em, start
			anchor={[yshift=-1ex]west}, end
			anchor={[yshift=1ex]west}]{}{h_{\algebra L}}
			\algebra L^\bullet \arrow[rr, shift left,
			"f:=F_1^1"] && \algebra H^\bullet \arrow[loop
			right, distance=3em, start
			anchor={[yshift=1ex]east}, end
			anchor={[yshift=-1ex]east}]{}{h_{\algebra
					H}} \arrow[ll, shift left, "g"]
		\end{tikzcd}
	\end{equation}
	where $g$, $h_\algebra{A}$ and $h_\algebra{B}$ are
	$\lie{g}$-equivariant and we have
	\begin{equation}
		g(\psi(\xi)))
		=
		\phi(\xi)
		\quad
		\text{and}
		\quad
		h_\algebra{L}(\phi(\xi))
		=
		h_\algebra{H}(\psi(\xi))
		=
		0
	\end{equation}
	for all $\xi\in \lie{g}$.  Then $F$ admits a $\lie{g}$-adapted
	$L_\infty$-quasi-inverse $G\colon \algebra{H}^\bullet\to
	\algebra{L}^\bullet$, such that $G_1^1=g$.
\end{theorem}
\begin{proof}
    Recall that with an homotopy equivalence we mean, that $$\id-f\circ g=[Q_{\algebra{H},1}^1,h_\algebra{H}] \text{ and }  
    \id-g\circ f=[Q_{\algebra{L},1}^1,h_\algebra{L}].$$
    The proof consists of two steps: first we
	construct a $\lie{g}$-adapted $L_\infty$-algebra structure
	$\widetilde{Q}_\algebra{H}$ on $\algebra H$ together with a
	$\lie{g}$-adapted $L_\infty$-morphism $\widetilde{G}\colon
	(\algebra{H},\widetilde{Q}_\algebra{H})\to
	(\algebra{L},Q_\algebra{L})$. As a second step we construct
	a $\lie{g}$-adapted $L_\infty$-isomorphism
	$(\algebra{H},\widetilde{Q}_\algebra{H})\to
	(\algebra{H},Q_\algebra{H})$.  We proceed by induction on the
	Taylor coefficients starting at $k=2$. Let us choose
	\begin{align}
		\widetilde{Q}_{\algebra{H},2}^1
		=
		Q_{\algebra{H},2}^1-
		f\circ(g\circ Q_{\algebra{H},2}^1-Q_{\algebra{L},2}^1\circ g^{\vee 2})
	\end{align}
	and 
	\begin{align}
		\widetilde{G}^1_2
		=
		h_\algebra{L}\circ (g\circ 
		Q_{\algebra{H},2}^1-Q_{\algebra{L},2}^1\circ g^{\vee 2}))
	\end{align}
	It is easy to see that these two maps fulfill the requirements
	up to order $2$ and are $\lie{g}$-adapted.
	Assume now that the $k$-th first Taylor
	coefficients of $\widetilde{Q}_\algebra{H}$ and
	$\widetilde{G}$ up to order $k$, for $k\geq 2$.  We denote by
	\begin{align*}
		\mu=(\widetilde{Q}_{\algebra{H}}
		\widetilde{Q}_{\algebra{H}})^1_{k+1} 
		\ \text{ and } \
		\nu = (\widetilde{G}\circ \widetilde{Q}_\algebra{H}-
		Q_\algebra{L}\circ \widetilde{G})^1_{k+1},
	\end{align*}
	then one can check that 
	\begin{align*}
		Q^1_{\algebra H, 1}\circ \mu - \mu\circ Q^{k+1}_{\algebra H, k+1}=0 \ \text{ and } \
		Q^1_{\algebra L, 1}\circ \nu + \nu\circ Q^{k+1}_{\algebra H, k+1}=g\circ \mu
	\end{align*}
	by using that $\widetilde{Q}_\mathscr{H}$ is a codifferential
	and $\widetilde{G}$ is an $L_\infty$-morphism up to order
	$k$. Thus we can show that
	\begin{align*}
		\widetilde{Q}^1_{\algebra H, k+1}:= 
		-h_\algebra{H}\circ \mu -2f\circ \nu +f\circ g\circ(h_\algebra{H}\mu+f\circ \nu) 
	\end{align*}
	and 
	\begin{align*}
		\widetilde{G}^1_{k+1}:=h_\algebra{L}\circ (\nu-g\circ(h_\algebra{H}\mu+f\circ \nu))
	\end{align*}
	are a codifferential and an $L_\infty$-morphism up to order
	$k+1$, resp.
	The formulas for the $L_\infty$-structure and morphism are not guessed, they can be extracted from 
	\cite{petersen}.
	 Note that the maps $\widetilde{Q}^1_{\algebra H, k+1}$
	and $ \widetilde{G}^1_{k+1}$ are equivariant, since they are
	just compositions of equivariant maps. Moreover, let $\xi\in
	\lie{g}$ and $x_i\in \algebra{H}[1]$, then
	\begin{align*}
		\mu(\psi(\xi)\vee x_1\vee\cdots\vee x_k)=
		\widetilde{Q}_{\algebra{H},k}^1(\xi\acts (x_1\vee\cdots \vee 
		x_k)- \xi\acts\widetilde{Q}_{\algebra{H},k}^1( (x_1\vee\cdots 
		\vee x_k)=0
	\end{align*}
	and 
	\begin{align*}
		\nu(\psi(\xi)\vee x_1\vee \cdots \vee x_k)
		=\widetilde{G}^1_{k}(\xi\acts (x_1\vee\cdots \vee 
		x_k))- \xi\acts \widetilde{G}^1_{k}(x_1\vee\cdots \vee 
		x_k)=0.
	\end{align*}
	This shows that we can extend by one the order of coderivation and of the
	$L_\infty$-morphism and the claim follows.  Using
	now $F$, we obtain a $\lie{g}$-adapted morphism $F\circ
	\widetilde{G}\colon (\algebra{H},\widetilde{Q}_\algebra{H})\to
	(\algebra{H},Q_\algebra{H})$, which has as a first Taylor
	coefficient $f\circ g$, so it is in general not an
	isomorphism.  Nevertheless, we can use $h_\algebra{H}\colon
	\algebra{H}^\bullet \to \algebra{H}^{\bullet-1}$ as a gauge
	transformation, i.e.  we consider the differential equation
	\begin{align*}
		\frac{\D}{\D t} \Phi_t = Q^{\Phi_t}(h_\algebra{H})    
	\end{align*}
	in the convolution $L_\infty$-algebra $\mathrm{Hom}(\cc{\Sym}
	(\algebra H[1]),\algebra{H})$ with the initial condition 
	$\Phi_0=F\circ \widetilde{G}$. Evaluated on $\psi(\xi)\vee 
	x_1\vee\cdots\vee x_k$ we get 
	\begin{align*}
		\frac{\D}{\D t} \Phi_{t,k+1}^1(\psi(\xi)\vee x_1\vee\cdots\vee x_{k})&=\sum_{i}\pm
		Q^1_{\algebra{H},j+1}(\Phi_{t,k}^j(\psi(\xi)\vee x_1\vee\overset{\wedge_i}\cdots\vee x_{k+1})\vee h_\algebra{H}(x_i))\\&
		+
		h_\algebra{H}( \widetilde{Q}^1_{\algebra{H},k+1}(\psi(\xi)\vee 
		x_1\vee\cdots\vee x_k)).
	\end{align*}
	For $k=0$, we get immediately $\frac{\D}{\D t} \Phi_{t,1}
	^1(\psi(\xi))=0$, since $h_\algebra{H}(\psi(\xi))=0$, and 
	therefore, since $\Phi_{0,1}^1(\psi(\xi)))=\psi(\xi)$. For 
	$k=1$, we get 
	\begin{align}
		\frac{\D}{\D t} \Phi_{t,1}^1(\psi(\xi)\vee x_1)=
		-\xi\acts h_\algebra{H}(x_1)+h_\algebra{H}(\xi \acts x_1)=0,
	\end{align}
	since $h_\algebra{H}$ is equivariant. Thus, $\Phi_{t,2}
	^1(\psi(\xi)\vee x_1)=0$, since $\Phi_{0,2}^1(\psi(\xi)\vee 
	x_1)=0$. We inductively get that $\Phi_{t,k}
	^1(\psi(\xi)\vee\cdot)=0$ for $k\geq 2$ and thus $\Phi_t$ is 
	$\lie{g}$-adapted. Moreover, one obtains easily 
	\begin{align*}
		\frac{\D}{\D t} \Phi_{t,1}^1(x)=
		Q^1_{\algebra{H},1}(h_\algebra{H}(x))+h_\algebra{H}
		((Q^1_{\algebra{H},1}(x))
	\end{align*}
	And thus $\Phi_{t,1}^1(x)=f\circ g +t(Q^1_{\algebra{H},1}
	(h_\algebra{H}(x))+h_\algebra{H}((Q^1_{\algebra{H},1}(x)))$ and 
	henceforth 
	\begin{align*}
		\Phi_{1,1}^1(x)=x.
	\end{align*}
	Therefore $\Phi_1$ is invertible as a $\lie{g}$-adapted 
	$L_\infty$-morphism, which we denote simply by $\Phi$. We define 
	$G=\widetilde{G}\circ \Phi^{-1}$. Then we see immediately that we have
	\begin{align*}
		F\circ G=F \circ \widetilde{G}\circ \Phi^{-1}\simeq \Phi 
		\circ\Phi^{-1}=\mathrm{id}.
	\end{align*}
	Note that here the homotopy is understood as $\lie{g}$-adapted
	homotopy equivalences. To show that $G\circ F$ is homotopic to
	identity (as $\lie{g}$-adapted $L_\infty$-morphisms), we apply
	what we did so far to the $L_\infty$-morphism $G$ and obtain a
	$\lie{g}$-adapted $L_\infty$-morphism $\hat{F}\colon
	\algebra{L}\to \algebra{H}$ such that $G\circ \hat{F}\simeq
	\mathrm{id}$, then we have
	\begin{align*}
		\hat{F}\simeq F\circ G\circ \hat{F}\simeq F
	\end{align*}
	which are again homotopy equivalences as $\lie{g}$-adapted
	$L_\infty$-morphisms.
	 Finally we obtain
	\begin{align*}
		G\circ F\simeq G\circ \hat{F}\simeq \mathrm{id}.
	\end{align*} 
\end{proof}

\section{Equivariant formality}
\label{sec:formality}

\subsection{A $\Der(A)$-adapted zig zag}

In the last section we have seen that we can construct a rather
explicit zig-zag of $G_\infty$ (and therefore
$L_\infty$)-algebras. Our plan is now to use Theorem \ref{Thm:gQuis}
in order to find explicit quasi-inverses which are
$\lie{g}$-adapted. Therefore, we need to show that the
$L_\infty$-quasi-isomorphisms are $\lie{g}$-adapted and find
appropriate homotopy equivalences of the underlying cochain complexes.
We first need to check that our constructed zig-zag from Theorem
\ref{thm:zigzag} is adapted for some Lie algebra.  Let us consider
therefore the Lie algebra $\Der(A)$, so in particular we have
canonical equivariant maps
\begin{equation}
	\Der(A)\hookrightarrow \multivect^\bullet(A)
	\text{ and }
	\Der(A)\hookrightarrow C^\bullet(A).
\end{equation}
Note that the Gerstenhaber bracket of $C^\bullet(A)$ evaluated on
$\Der(A)$ is just the commutator, and thus $C^\bullet(A)$ is
$\Der(A)$-adapted. For $\multivect^\bullet(A)$ the statement is
trivial.
\begin{proposition}
	\label{prop:equiv}
	Let $A$ be a commutative algebra.
	\begin{propositionlist}
		\item The Lie algebra part of the Gerstenhaber algebras 
	$\comm(\mathscr{L}_{C^\bullet(A)}[-1])$,  
	$\comm (L_A[-1])$ and 
	$\comm(\mathscr{L}_{\multivect(A)}[-1])$ 
	are $\Der(A)$-adapted 
	via 
	$\comm^1(\colie^1(\multivect^\bullet(A)))
	=\multivect^\bullet(A)\supseteq \Der(A)$.
	\item The diagram 
		\begin{equation}
		\label{eq:thmzigzag}
		\begin{tikzcd}
			\comm(\mathscr{L}_{C^\bullet(A)}[-1]) & 
			\comm (L_A[-1]) \arrow[l]\arrow[r]&
			\comm(\mathscr{L}_{\multivect(A)}[-1])\arrow[d]\\
			C^\bullet(A)\arrow[u] & & 
			\multivect^\bullet(A)
		\end{tikzcd}
	\end{equation} 
	is a diagram of $\Der(A)$-adapted $L_\infty$--algebras.
	\end{propositionlist}

\end{proposition}
\begin{proof}
	Since $\comm (L_A[-1])$ is a dg sub-Gerstenhaber
	algebra of the other two, we only need to show that
	$\Der(A)\subseteq\comm (L_A[-1])$ is a sub-Lie algebra, which
	is clear by construction of $L_A$.  
	To prove the second item we first observe that all the morphisms in
	the diagram \eqref{eq:thmzigzag} are morphisms of Gerstenhaber algebras, which are canonically 
	$\Der(A)$ adapted, 
	besides the map $C^\bullet(A)\to
	\comm(\mathscr{L}_A[-1])$. This map is given in particular by the
	$L_\infty$-part of the concatenation of $\chi^{C^\bullet(A)}$
	and the map $\gerst(\gerst^\vee(C^\bullet(A)))\to
	\comm(\mathscr{L}_{C^\bullet(A)}[-1]))$ which comes from the
	projection \newline $\flie(\cocom
	(\mathscr{L}_{C^\bullet(A)}[1])[-1]\to
	\mathscr{L}_{C^\bullet(A)}$, see Proposition
	\ref{prop:dgliebi}. The Taylor coefficients of this
	$L_\infty$-morphism are given by
	\begin{equation}
		\begin{tikzcd}
			\cocom(C^\bullet(A)[2])
			\hookrightarrow
			\cocom(\mathscr{L}_{C^\bullet(A)}[1]) \arrow[r]& 
			\comm(\flie(\cocom(\mathscr{L}_{C^\bullet(A)}[1])[-1])[-1])[2]\arrow[d]\\
			& \comm(\mathscr{L}_{C^\bullet(A)}[-1])[2]
		\end{tikzcd}
	\end{equation}
	Here the horizontal arrow is the inclusion and the vertical
	one is the Gerstenhaber algebra map obtained in Proposition
	\ref{prop:dgliebi}. Thus the evaluation of the above
	composition at $\cocom^{\geq 2}(\mathscr{L}_{C^\bullet(A)})$
	is equal to zero
	and therefore the $L_\infty$-morphism is
	$\Der(A)$-adapted.
\end{proof}

Let us now take a closer look at the zig-zag from Theorem
\ref{thm:zigzag}. The first Proposition we can deduce is actually
rather general.
\begin{proposition}\label{Prop:HomotopyI}
	Let $A$ be a commutative algebra over a field of
	characteristic $0$, then the first Taylor coefficient of the
	$\Der(A)$-adapted $L_\infty$ morphism
	\begin{equation}
		\label{eq:homotopyret2}
		\begin{tikzcd} 
			\comm(\mathscr{L}_{C^\bullet(A)}[-1]) 
			&&  
			C^\bullet(A)
			\arrow[ll, " "]
		\end{tikzcd}
	\end{equation}
	fits into a special deformation retract
	\begin{equation}
		\label{eq:homotopyret2}
		\begin{tikzcd} 
			\arrow[loop left, distance=3em, start anchor={[yshift=-1ex]west}, end anchor={[yshift=1ex]west}]{}{h_{\comm(\mathscr{L}_{C^\bullet(A)}[-1]) }}
			\comm(\mathscr{L}_{C^\bullet(A)}[-1]) 
			\arrow[rr, shift left, " \pi"] 
			&&  
			C^\bullet(A)
			\arrow[ll, shift left, "i"]
		\end{tikzcd}
	\end{equation}
	such that for all derivations $D\in \Der(A)$ 
	\begin{align*}
		i(D)=(D),\ \ \pi((D))=D \ \text{ and } \
		h_{\comm(\mathscr{L}_{C^\bullet(A)}[-1]) }((D))=0 
	\end{align*}
	and the maps $i$, $\pi$ and $h_{\comm(\mathscr{L}_{C^\bullet(A)}[-1]) }$  are $\Der(A)$-invariant. 
\end{proposition}
\begin{proof}
	Consider the Hochschild complex $(C^\bullet(A), \partial) $. From Prop.\ref{prop:a6}
	we obtain the following special
	deformation retract
	\begin{equation}
	\label{eq:unpert}
		\begin{tikzcd}
			(C^\bullet(A), \partial) \arrow[rr, shift left, "i"]  &  &             (\comm(\colie(C^\bullet(A)[1])[-1]) , \delta+\partial)\arrow[ll,              shift left, "p"] \arrow[loop right,  	distance=4em, start anchor={[yshift=1ex]east}, end anchor={[yshift=-1ex]east}]{}{H} 
		\end{tikzcd} 
	\end{equation}
	Note that the differential on
	$\comm(\colie(C^\bullet(A)[1])[-1])=\comm(\mathscr{L}_{C^\bullet(A)}[-1])$
	in the above diagram is given by the Hochschild differential 
	$\partial$ together with the the Lie cobracket $\delta$.  
	This is not the differential $\D$ from Lemma \ref{lemma:liebialgcolie}.
	Recall that $\D$ encodes the Taylor
	coefficients of the $G_\infty$-structure on $C^\bullet(A)$ of
	the form
	\begin{align*}
		q^n\colon \colie^n(C^\bullet(A)[1])\to  C^\bullet(A)[1]
	\end{align*}
	where $q^1=\partial$. This means in particular that
        $\D-\partial$ is given by the structure maps $q^{\geq 2}$ and therefore
        lowers the total tensor degree in
        $\comm(\mathscr{L}_{C^\bullet(A)}[-1])$. Since $H\colon
        \comm(L_{C^\bullet(A)}[-1])\to \comm(L_{C^\bullet(A)}[-1])$
        keeps the total tensor degree, we have that
        $H\circ(\D-\partial)$ is locally nilpotent. Thus we can apply
        Prop.\ref{prop:HomologicalPerturbation} to obtain a new special deformation
        retract
	\begin{center}
		\begin{tikzcd}
			(C^\bullet(A), \partial) \arrow[rr, shift left, "\tilde{i}"]  &  &             (\comm(\colie(C^\bullet(A)[1])[-1]) , \delta+\D)\arrow[ll,              shift left, "\pi"] \arrow[loop right,  	distance=4em, start anchor={[yshift=1ex]east}, end anchor={[yshift=-1ex]east}]{}{\tilde{H}} 
		\end{tikzcd} 
	\end{center}
	The map $\tilde{i}$ is given by $i$, since $i$ is a chain map
        with respect to the differential and its perturbation, in fact
        we have $(\D-\partial)\circ i=0$.  
        For the same reason the differential
        $\partial$ on $C^\bullet(A)$ does not change. Moreover, for $D\in \Der(A)$,
        that
	\begin{align*}
		\tilde{H}((D)))=H(i(D))=0 \ \text{ and } \ 
		\pi((D))=\pi(i(D))=D. 
	\end{align*}
	Finally, every map of the retract \ref{eq:unpert} is invariant with
        respect to the action of $\Der(A)$ and so is the perturbation.
        This means, in particular, that the maps $i$, $\pi$ and
        $\tilde{H}$ are invariant with respect to $\Der(A)$.
\end{proof}
\begin{remark}\label{Rem:PropHom}
	Note that the construction of the deformation retract in
        Proposition \ref{Prop:HomotopyI} holds for every
        $G_\infty$-algebra fulfilling condition \ref{item:
          Liebial}. The adapted-ness does not have to make sense for
        general $G_\infty$-algebras obviously.
\end{remark}

\subsection{A $\lie g$-adapted formality map}

From now on we consider $A=\Cinfty(M)$, where $M$ is a smooth manifold
and let $G$ be a Lie group acting on $M$ via $\Phi\colon G\times M\to
M$ with corresponding Lie algebra action $\lie{g}\to
\Secinfty(TM)$. 
\begin{definition}[Equivariantly Projective]
	Given a Lie group action 
	$\Phi\colon G\times M \to M$, we say that $M$ is equivariantly 
	projective with respect to $\Phi$ if
	 there exists a vector space $V$ and a
	Lie group morphism $L\colon G\to \mathrm{GL}(V)$, together with vector bundle maps
	\begin{align*}
		I\colon TM\to M\times V \ \text{ and } \
		P\colon M\times V\to TM
	\end{align*}
	covering the identity, such that 
	\begin{align*}
	I\circ d\Phi_g= (\Phi_g\times L_g)\circ I, \ 
	P\circ (\Phi_g\times L_g)
	=d\Phi_g \circ P \ \text{ and } \ 
	P\circ I=\mathrm{id}, 
	\end{align*}
	where $d\Phi_g:TM\rightarrow TM$ denotes the differential of $\Phi_g$.
\end{definition}
\begin{remark}
\label{rem:epb}
	Suppose that the action of $G$ on $M$ is equivariantly projective and let $\{e_i\}_{i\in \mathscr{I}}$ be a linear basis of $V$ and $\{e^i\}_{i\in \mathscr{I}}$ its dual in $V^*$. We will consider $\{e_i\}_{i\in \mathscr{I}}$  and $\{e^i\}_{i\in \mathscr{I}}$ as sections of the trivial bundles
$M\times V\to M$ and $M\times V^*\to M$. Let us denote
by
\begin{align*}
	X_i= P(e_i) \in \Secinfty(TM) \ \text{ and } 
	\alpha^i =e^i\circ I \in \Secinfty(T^*M)
\end{align*}
for all $i\in \mathscr{I}$. 

If we denote by $(L_g)^i_j$ the coefficients of the linear map $L_g$ with respect to the basis 
$e_i$, then we have 
\begin{align*}
	(\Phi_g)_*X_i=(L_g)^j_iX_j \ \text{ and } \ 
	\Phi_g^*\alpha^k= (L_g)^k_\ell \alpha^\ell
\end{align*}
for all $i,k\in \mathscr{I}$. Note that moreover, we have 
for all vector fields $X\in \Secinfty(TM)$ that 
\begin{align*}
	X=\alpha^i(X)X_i
\end{align*}
by construction, i.e. $\{X_i\}_{i\in \mathscr{I}}$ and $\{\alpha^i\}_{i\in \mathscr{I}}$ form a projective dual basis. 
\end{remark}
For future reference let us record the following.
\begin{proposition}
Let $\Phi\colon G\times M \to M$ be a Lie group action.
\begin{propositionlist} 
\item If  $G$ is a compact group and $M$ is a compact manifold, the action $\Phi$ is equivariantly projective;
\item If $\Phi$ is free and proper action of a second countable locally compact  Lie group on a second countable smooth manifold, then it is equivariantly projective. 
\end{propositionlist}
\end{proposition}
\begin{proof}
The claim about actions of compact groups is the content of the theorem of Segal, see \cite{Segal}. If the action is free and proper, a theorem of Palais says that $M$ is a principal $G$-bundle. In particular $M/G$ is a finite dimensional smooth manifold and,  locally, $M$ is of the form $\mathcal{O}\times G$, where  $\mathcal{O}$ is an open subset of $M/G$ with trivial tangent bundle. Since $M/G$ is finite dimensional, given any covering $\mathcal{O}_i$ with $i\in I$ of $M/G$, there exists a finite covering $\{\mathcal{U}_1,\ldots,\mathcal{U}_N\}$ such that each $\mathcal{U}_i$ is a union of a disjoint family of subsets of $\mathcal{O}_i$. In particular, $M$ admits a {\it finite} $G$-equivariant covering  by opens subsets of the form  $\mathcal{O}\times G$ with $\mathcal{O}$ diffeomorphic to an open (generally not connected) open subset of $\mathbb{R}^n$. Since $G$ is parallelisable, the standard partition of unity completes the proof. 
\end{proof}
\begin{remark}
\begin{propositionlist}
\item The above cases can be subsummed under the heading of proper action, and it seems that a proof can be constructed using the methods of the above two cases and Palais' equivariant slice theorem, but the details still need to be worked out.
\item One can also replace algebraic tensor products with completed, say projective, tensor products, since all the spaces appearing have naturally a locally convex nuclear topology (to be more precise these are test functions, differential operators and distributions) in the constructions we are preforming. 
Again in this case the associated notion of ``equivariant projectivity'' follows immediately from the equivariant embedding theorem into a Hilbert space, again due to Palais. 
\end{propositionlist} 
\end{remark}
\begin{proposition}\label{Prop:HomotopyII}
	Let $G$ be a Lie group acting on a smooth manifold $M$, such
        that $M$ is equivariantly projective with respect to this
        action.  Then the Gerstenhaber morphism $f\colon \comm
        (L_{\Cinfty(M)}[-1])\to \multivect^\bullet(M)$ from the
        diagram
	\begin{center}
		\begin{tikzcd}
			\comm (L_{\Cinfty(M)}[-1]) \arrow[dr, "f"']\arrow[r]&
			\comm(\mathscr{L}_{\multivect(M)}[-1])\arrow[d]\\
			& 
			\multivect^\bullet(M)
		\end{tikzcd}
	\end{center} 
	fits
	into a chain homotopy equivalence
	\begin{equation}
		\label{eq:homotopyret}
		\begin{tikzcd} 
			\arrow[loop left, distance=3em, start anchor={[yshift=-1ex]west}, end anchor={[yshift=1ex]west}]{}{h_{\comm (L_{\Cinfty(M)}[-1])}}
			\comm (L_{\Cinfty(M)}[-1])
			\arrow[rr, shift left, "f"] 
			&&  
			\multivect^\bullet(M)
			\arrow[ll, shift left, "g"] 
		\end{tikzcd}
	\end{equation}
	such that $f$, $g$ and $h_{\comm (L_{\Cinfty(M)}[-1])}$ are $G$-equivariant and we have 
	\begin{align*}
		f((\xi_M)) =\xi_M, \ g(\xi_M)=(\xi_M) \ \text{ and }
		h_{\comm (L_{\Cinfty(M)}[-1])}((\xi_M))=0
	\end{align*}
\end{proposition}
\begin{proof}
	As discussed in \cite{dolgushevetal:2007}, we note that $L_{\Cinfty(M)}$ is the sum of $\colie(\Cinfty(M)[1])$ and the cofree Lie comodule 
	$T^c\Cinfty(M)[1]\tensor \multivect^1(M)$, i.e. 
	$$
	L_{\Cinfty(M)}=\colie(\Cinfty(M)[1])\oplus
        T^c\Cinfty(M)[1]\tensor \multivect^1(M).
    $$ 
    From Remark \ref{rem:epb} we know that we
        can find a projective dual basis $\{X_i,\alpha_i\}_{i\in
          \mathscr{I}}$ on which the group acts linear. 
          Thus we can use the deformation
        retract \eqref{DefRetract: projMod}, which looks in our case as
	\begin{center} 
		\begin{tikzcd}
			\multivect^\bullet(M)\arrow[d, equal]& &\comm(L_{\Cinfty(M)}[-1])\arrow[d,equal]\\\Sym_{\Cinfty(M)}\mathscr{\multivect}^1(M)[-1] 
			\arrow[rr, shift left, "I"]  &  &
			\comm(\colie(\Cinfty(M)[1])[-1]\oplus T^c\Cinfty(M)[1]\tensor \multivect^1(M)[-1])
			\arrow[ll, 
			shift left, "\Pi"] 
			\arrow[out=240,in=300,loop,swap, "H"]\\
			
		\end{tikzcd}
	\end{center}
	Note that by construction all the maps are
        $G$-invariant. Moreover, we have that $\Pi$ coincides with
        $f$, as discussed in the proof of Prop.~\ref{prop:a8}
        We observe that we have:
	\begin{align*}
		I(\xi_M)=(\alpha^i(\xi_M))(X_i).
	\end{align*}
	Let us define the map $g\colon \multivect^\bullet(M)\to \comm(L_{\Cinfty(M)}[-1])$ by setting
	\begin{align*}
		g(Y)
		=(Y)
	\end{align*}
	for $Y\in \multivect^1(M)$, and $g\at{\multivect^k(M)}=I\at{\multivect^k(M)}$ for $k\neq
        1$. Then we clearly have that $\Pi\circ g=\mathrm{id}$ and
        that $g$ is an $G$-equivariant cochain map, because $G$ acts
        linearly of $X_i$ and $\alpha^i$. Let us define
	\begin{align*}
		h_{\comm (L_{\Cinfty(M)}[-1])}=H-H\circ g\circ \Pi
	\end{align*}
	which is $G$-invariant and moreover, we have that 
	\begin{align*}
		h_{\comm (L_{\Cinfty(M)}[-1])}(\xi_M)=
		H(\xi_M)-H(g(\Pi((\xi_M))))=H(\xi_M)-H(\xi_M)=0 
	\end{align*}
	which concludes the proof. 
\end{proof}

Now we have all the ingredients together to prove the main theorem

\begin{theorem}\label{Thm: ExadaptedLinfty}
	Let $G$ be a Lie group acting on a smooth manifold M, such
        that $M$ is equivariantly projective with respect to this
        action. Then there exists a $\lie{g}$-adapted
        $L_\infty$-morphism
	\begin{align*}
		F\colon T_\poly(M)\to D_\poly(M),
	\end{align*}
	such that the first Taylor component $F_1^1$ is the Hochschild-Kostant-Rosenberg map $\mathrm{hkr}$.
	Moreover, the $F$ can be constructed in a way that its Taylor coefficients are local maps. 
\end{theorem}

\begin{proof}
	Summing up the results from	
	 the Propositions \ref{prop:equiv}, \ref{Prop:HomotopyI} and
        \ref{Prop:HomotopyII} in the case $A=\Cinfty(M)$, we obtain
	the following diagram:
	\begin{center}
		\begin{tikzcd}
			\comm(\mathscr{L}_{C^\bullet(\Cinfty(M))}[-1]) [1]\arrow[d, shift left,  "P"] & 
			\comm (L_{\Cinfty(M)}[-1])[1] \arrow[l]\arrow[r]&
			\comm(\mathscr{L}_{\multivect(M)}[-1])[1]\arrow[d]\\
			C^\bullet(\Cinfty(M))[1]\arrow[u, shift left]& & 
			\multivect^\bullet(M)\arrow[ul, "G"]
		\end{tikzcd}
	\end{center} 
	where all the maps are $\lie{g}$-adapted $L_\infty$-morphisms.
	This implies that  we obtain a $\lie{g}$-adapted $L_\infty$-morphism
	$\tilde{F}\colon T_\poly(M)\to C^\bullet(M)[1]$
	which can be immediately upgraded to a $\lie{g}$-adapted $L_\infty$-morphism
	$
		F\colon T_\poly(M)\to D_\poly(M)
	$.
	This is possible as $D_\poly(M)[-1]$ is a sub-$G_\infty$ algebra of
$C^\bullet(\Cinfty(M))$ containing the fundamental vector fields of
the Lie group action, see Remark~\ref{Rem:PropHom}. 
 $P_1$ coincides with the map $\pi$ from
Proposition \ref{Prop:HomotopyI}: for $D_1,\dots, D_k$, one obtains
that
\begin{align*}
	\pi((D_1)\cdots (D_k))=\frac{1}{k!}\sum_{\sigma\in S_k}D_{\sigma(1)}\cup 
	\cdots \cup D_{\sigma(k)}
\end{align*}
by using eq. \eqref{Eq: Homotopy} and the
fact that the lowest order of the $G_\infty$-structure of
$D_\poly(M)$ is the symmetrization of
the cup product. This means in particular that the first Taylor
coefficient of $F$ is exactly the Hochschild-Kostant-Rosenberg map. Note that 
because of the naturality of the deformation retracts in Prop.\ref{Prop:HomotopyI} and 
and Prop.\ref{Prop:HomotopyII}, we see immediately, that all the Taylor coefficients of $F$ are local maps.  
\end{proof}

\begin{remark}
	It should be easy to see that the Taylor coefficients of the
        $L_\infty$-morphism are differential operators, since the
        structure maps of the $G_\infty$ are, at least when restricted
        to $D_\poly(M)$, since they are just linear combinations of
        the braces defined on $D_\poly(M)$.
\end{remark}
Finally we obtain
\begin{theorem}
	\label{Thm: ExQuantumMoment}
		Let $G$ be a Lie group acting on a manifold $M$ such that $M$
        is equivariantly projective with respect to this action. Then
        there is a one-to-one correspondence between equivalence
        classes of $G$-invariant star product with quantum momentum maps
        and equivalence classes of $G$-invariant formal Poisson
        structures with formal momentum maps.
\end{theorem}

\begin{proof}
	Let $
		F\colon T_\poly(M)\to D_\poly(M)
	$ be the $\lie{g}$-adapted morphism we have constructed in Theorem \ref{Thm: ExadaptedLinfty}, 
	having $F_1=\mathrm{hkr}$.
	Using the
	results of \cite[Theorem 6.16]{dippell:2024}, we can find $G$-equivariant maps
	$h\colon D_\poly(M)\to D_\poly(M)$ and $p\colon D_\poly(M)\to
	T_{\poly}(M)$, such that
	\begin{center}
		\begin{tikzcd}
			T_\poly(M) \arrow[rr, shift left, "\mathrm{hkr}"]   
			&       &      (D_\poly(M), \partial)\arrow[ll, shift left, "p"] 
			\arrow[loop right, distance=4em, start anchor={[yshift=1ex]east}, end anchor={[yshift=-1ex]east}]{}{h} 
		\end{tikzcd} 
	\end{center}
	is a special deformation retract and such that the maps $\mathrm{hkr}$, $p$ and
	$h$ satisfy the hypothesis of Theorem \ref{Thm:gQuis}. Therefore,
	we can construct a $\lie{g}$-adapted quasi-inverse $G$ of the map
	$F$. Moreover, we can achieve $G\circ F=\mathrm{id}$ and the homotopy
	between $F\circ G$ and the identity can be chosen to be
	$\lie{g}$-adapted, as discussed in Theorem \ref{Thm:gQuis}.
	Consider now the
	curved Lie algebras
	$
		\big(T_\lie{g}(M)[[\hbar]],\hbar^2\lambda,0, [\cdot,\cdot]_\lie{g}\big)
	$
	and 
	$
		\big(D_\lie{g}(M)[[\hbar]],\hbar^2\lambda, \partial_\lie{g}, [\cdot,\cdot]_\lie{g}\big)
	$
	which are $\hbar$-rescaled versions of the corresponding curved Lie
	algebras from Definitions~\ref{def:equiv1}, \ref{def:equiv2}, resp. By Proposition
	\ref{prop:CompatibilityMorphismCurvatur} we can extend the $\lie{g}$-adapted morphism 
	$F$ to a curved
	$L_\infty$-morphism
	\begin{align}\label{Eq: scaledcurvedMorph}
		F^\lie{g}\colon T_\poly(M)[[\hbar]]\to D_\poly(M)[[\hbar]].
	\end{align}
	Note that $G$-invariant formal Poisson structures with formal
	momentum maps are exactly Maurer-Cartan elements of $
	\big(T_\lie{g}(M)[[\hbar],\hbar^2\lambda,0,
	[\cdot,\cdot]_\lie{g}\big)$ and $G$-invariant star product
	with quantum momentum maps are Maurer-Cartan elements of
	$\big(D_\lie{g}(M)[[\hbar]],\hbar^2\lambda,
	\partial_\lie{g}, [\cdot,\cdot]_\lie{g}\big)$ and thus,
	since $F^\lie{g}$ is a homotopy equivalence, the claim
	follows directly.
\end{proof}

%
%
\section*{Outlook}
\label{sec:outlook}

The above results open up many directions for future research.  Here,
we outline the two most immediate directions in more detail.


The first question concerns the
\emph{quantization-reduction diagram} in the realm of deformation
quantization of Poisson manifolds.  The existence of a quantum momentum
map obtained for a Poisson structure $\pi$ with a momentum map $J$
implies that one can perform a BRST-like reduction for star products
(see \cite{bordemann.herbig.waldmann:2000a}).  The question which is
still to be answered is what the relation of the reduced star product
and the reduced Poisson structure is.
This question can now be completely formulated in terms of curved
$L_\infty$ morphisms: let $J\in (\lie{g}^*\tensor \Cinfty(M))$ be an
equivariant map, such that $0\in \lie{g}^*$ is a regular value of the
induced map $M\to\lie{g}^*$. Moreover, we consider now the curved
$L_\infty$-morphism $$ F^\lie{g}\colon
\big(T_\lie{g}(M)[[\hbar]],\hbar\lambda,0, [\cdot,\cdot]_\lie{g}\big)
\longrightarrow \big(D_\lie{g}(M)[[\hbar]],\hbar\lambda,
\partial_\lie{g}, [\cdot,\cdot]_\lie{g}\big) $$ we constructed in this paper.
We can twist this morphism by $-J$ to obtain a curved morphism (see
\cite{kraft:2024} for details of twisting of curved morphisms):
\begin{align*}
	F^\lie{g}_J\colon 
	\big(T_\lie{g}(M)[[\hbar]],\hbar\lambda,-[J,\cdot], [\cdot,\cdot]_\lie{g}\big)
	\to 
	\big(D_\lie{g}(M)[[\hbar]],\hbar\lambda,
	\partial_\lie{g}-[J,\cdot], [\cdot,\cdot]_\lie{g}\big).
\end{align*}
The corresponding Maurer--Cartan elements are now invariant 
formal Poisson structures with formal momentum maps
shaving $J$ as the zero order and invariant star products with 
quantum momentum maps fulfilling the same property, respectively.  
Here one should notice that 
\begin{align*}
	F^\lie{g}_k(J^{\vee k})=0
\end{align*}
for $k\geq 2$, since $F^\lie{g}_k(J^{\vee k})\in
(\Sym^{k}\lie{g}^*\tensor D_\poly^{1-2k}(M))^G$. For the same reason
the twisted morphism is well-defined.
It is important to remark that the construction
of our $\lie{g}$-adapted morphism implicitly depends on the choice of
a Drinfel'd associator, since the $G_\infty$-structure on $D_\poly(M)$
does.  Using the reduction morphisms $T_\red$ and $D_\red$ from
\cite{esposito:2022,esposito:2023}, we obtain the following
conjecture.
\begin{conjecture}
	Let $F\colon T_\poly(M)\to D_\poly(M)$ be a $\lie{g}$-adapted
        $L_\infty$ morphism and let $G\colon T_\poly(M_\red)\to
        D_\poly(M_\red)$ be the $L_\infty$-quasi-isomorphism obtained by the zig-zag in \cite{dolgushevetal:2007},
         both constructed
        with respect to the same Drinfel'd associator, then the
        diagram of curved $L_\infty$-morphisms
	\begin{center}
		\begin{tikzcd}
			\big(T_\lie{g}(M)[[\hbar],\hbar\lambda,-[J,\cdot], [\cdot,\cdot]_\lie{g}\big)
			\arrow[rr, "F^\lie{g}_J"] 
			\arrow[d, "T_\red"]
			& &
			\big(D_\lie{g}(M)[[\hbar],\hbar\lambda,
			\partial_\lie{g}-[J,\cdot], [\cdot,\cdot]_\lie{g}\big)
			\arrow[d, "D_\red"]\\
			(T_\poly(M_\red)[[\hbar]],[\cdot,\cdot])\arrow[rr, "G"]
			& &
			(D_\poly(M_\red)[[\hbar]],\partial, [\cdot,\cdot])
		\end{tikzcd}
	\end{center}
	commutes up to homotopy. 
\end{conjecture}
Moreover, we believe that the diagram above can shed new light on the \emph{quantum anomalies}
which appeared in \cite{cattaneo} for general coisotropic submanifolds. 

The second natural question concerns formality for chains, that has
been proven in \cite{shoikhet}. More precisely, one may wonder wether
the formality map obtained for Hochschild chains can be lifted to
equivariant Hochschild chains and equivariant differential forms. This
would open up the study of equivariant algebraic index theory. We
believe that a good starting point to generalize our results to chains
would be proving formality for chains using the dequantization functor
from \cite{enriquez:2005}, from the point of view developed in
\cite{severa}.
 Nevertheless, these work will be carried out in future projects.

\begin{appendix}

\section{Cofree Lie coalgebras}

\subsection{The dual of the Poincaré-Birkhoff-Witt Theorem for cofree conilpotent Lie coalgebras}
	
In this appendix we discuss the dual version of the
Poincaré-Birkhoff-Witt theorem for the cofree Lie coalgebra cogenerated by a graded
vector space. Note that there are proofs for locally finite Lie coalgebras
(see \cite{michaelis,block}) but none of the existing results is
suitable for our needs, so we give a proof here. The content of this subsection is
basically included in
\cite{fresse}, but for convenience we make it precise here for conilpotent bialgebras.
	
Let $(B,\mu,\iota, \Delta, \epsilon)$ be a graded conilpotent
commutative bialgebra over a ring $\Bbbk$ of characteristic $0$.  We
can define the convolution product on $\star$ on $\mathrm{Hom}(B,B)$
by
\begin{align*}
	f\star g
	=
	\mu \circ (f\tensor g)\circ \Delta 
\end{align*}
which turns $\mathrm{Hom}(B,B)$ into a graded algebra. Note that the
unit of $(\mathrm{Hom}(B,B),\star)$ is given by $\iota\circ \epsilon$.
One can easily compute the following identity
\begin{align}\label{Eq:CompConvConc}
	\id^{\star k}\circ \id^{\star \ell}
	=
	\id^{\star k\ell}
\end{align}
for $k,\ell\geq 0$, where we interpret $H^{\star 0}=\iota\circ \epsilon$
for all $H\in \mathrm{Hom}(B,B)$.  Let us define the \emph{ first
Eulerian idempotent} by
\begin{align*}
	e^{(1)}=\log_\star(\id)=\sum_{k=1}^\infty \frac{(-1)^{k+1}}{k}(\id-\iota\circ \epsilon)^{\star k}=
	\sum_{k=1}^\infty \frac{(-1)^{k+1}}{k} P^{\star k},
\end{align*} 
where in the last equality we set $P:=\id- \iota\circ\epsilon$. 
Note that $e^{(1)}$ is well-defined, since $B$ is conilpotent.
\begin{lemma}
	The first Eulerian idempotent $e^{(1)}\in
            \mathrm{Hom}(B,B)$ is and idempotent,
            i.e. $e^{(1)}\circ e^{(1)}=e^{(1)}$.
\end{lemma}
\begin{proof}
	Let us write $$e^{(1)}=\sum_{k=1}^\infty
            \frac{(-1)^{k+1}}{k}\sum_{i=0}^k\binom{k}{i}(-1)^{k-i}\id^{\star
              i},$$ which means we have using Equation
            \eqref{Eq:CompConvConc}
	\begin{align*}
		e^{(1)}\circ e^{(1)}&=
		\sum_{k=1}^\infty \sum_{l=1}^\infty\frac{(-1)^{k+l+2}}{kl}\sum_{i=0}^k\sum_{j=0}^l
		\binom{l}{j}\binom{k}{i}(-1)^{k+l-j-i}\id^{\star ij}\\&
		=
		\sum_{k=1}^\infty\frac{(-1)^{k+1}}{k}
		\sum_{i=0}^k\binom{k}{i}(-1)^{k-i}\log_\star(\id^{\star i})\\&
		=
		\sum_{k=1}^\infty\frac{(-1)^{k+1}}{k}
		\sum_{i=0}^k\binom{k}{i}(-1)^{k-i}i\log_\star(\id)\\&
		=
		\sum_{k=1}^\infty(-1)^{k+1}(1-1)^{k-1}\log_\star(\id)\\&
		=\log_\star(\id)=e^{(1)}
	\end{align*}
\end{proof}
Note that now we have that 
\begin{align*}
	\id
	=
	\exp_{\star}({\log_\star(\id)})=\exp_{\star}({e^{(1)}})
	=
	\sum_{k=0}^\infty \frac{(e^{(1)})^{\star k}}{k!}
\end{align*}
where we define now the $k$-th Eulerian idempotent by
$e^{(k)}=\frac{(e^{(1)})^{\star k}}{k!}$.  Let us now take a
closer look to $\id^{\star k}$ for some $k\geq 2$. We have
that
\begin{align*}
	\id^{\star k}=\exp_\star(\log_\star(\id^{\star k}))= \exp_\star(k\log_\star(\id))=\sum_{i=0}^\infty k^ie^{(i)}. 
\end{align*}
Using now Equation~\eqref{Eq:CompConvConc} 
 we obtain the following claim.
\begin{theorem}
	The set $\{e^{(k)}\}_{k\in \mathbb{N}_0}$ is a
    complete set of orthogonal projections,
    i.e. $e^{(k)}\circ e^{(\ell)}=\delta^{kl}e^{(k)}$ and
    $\id=\sum_{k=0}^\infty e^{(k)}$.
\end{theorem}
\begin{example}
Let $V$ be a vector space and let $\comm (V)=\bigoplus_{k\geq 1}\bigvee^k V$ be the 
free commutative algebra generatated by $V$. It is well-known that it can be endowed 
with a cocommutative bialgebra structure using the shuffle coproduct.  It can be shown that 
he $k$-th Eulerian idempotent 
is now given by the projection $\comm (V)\to \bigvee^k V$.
\end{example}
We denote from now on $B^{(k)}=e^{(k)}(B)$ and we get that
$B=\bigoplus_{k\geq 0} B^{(k)}$.  Note that this means that
$x\in \mathrm{image}( e^{(\ell)})$ if and only if $\id^{\star
  k}(x)=k^\ell x$. Moreover, since the algebra structure of $B$ is
commutative the convolution product of two algebra morphisms is
again an algebra morphism. So we obtain that $\id^{\star k}$ is
an algebra morphism, which implies that
\begin{align}
	\label{Eq: GradProd}
	B^{(\alpha_1)}\cdots B^{(\alpha_k)}\subseteq B^{(\sum_{i=1}^k\alpha_i)}
\end{align}
for all $(\alpha_1,\cdots,\alpha_k)\in \mathbb{N}_0^k$.
In particular we have for all $x_1,\dots, x_k\in B^{(1)}$ that the equation 
$$e^{(k)}(x_1\cdots x_k)=x_1\cdots x_k.$$
holds. Let us now consider the coproduct: for $x\in B^{(1)}$, we have that 
\begin{align*}
	\Delta(x)-x\tensor 1- 1\tensor x \subseteq \bigoplus_{i + j\geq 2} B^{(i)}\tensor B^{(j)},
\end{align*}
because applying $\epsilon\tensor \id$ and
$\id\tensor\epsilon$ on the left side it vanishes, so it has
to vanish on the right side.  This immediately implies that, by
using Equation~\eqref{Eq: GradProd}, for
$x_1,\dots,x_k\in B^{(1)}$ we get
\begin{align}\label{Eq:CoProdGrading}
	\Delta(x_1\cdots x_k)-\prod_{i=1}^k (x_i\tensor 1+1\tensor x_i) \in \bigoplus_{i+j\geq k+1} B^{(i)}\tensor B^{(j)}.
\end{align}
If we introduce now the decreasing filtration
\begin{align}
	\label{Def:filtration}
	B^{(\geq k)}
	:=
	\bigoplus_{i\geq k} B^{(i)},
\end{align}	 
we see that $(B,\Delta)$ is a filtered coalgebra,
i.e. $\Delta(B^{(\geq k)})\subseteq \sum_{i+j=k}B^{(\geq
  i)}\tensor B^{(\geq j)}$, since $B^{(k)}$ is generated as a
vector space by $k$-fold multiplications of elements of
$B^{(1)}$ by the definition of $e^{(k)}$.  Using Equation~\eqref{Eq: GradProd}
and that $\Delta$ is an algebra morphism, we obtain
that
\begin{align*}
	\Delta^{n-1} (x_1 \cdots x_n)
	= 
	\sum_{\sigma\in S_n}\chi(\sigma) x_{\sigma(1)}\tensor\cdots \tensor x_{\sigma(k)} 
	+ \bigoplus_{i_1+\cdots i_n \geq n+1}B^{(i_1)}\tensor \cdots \tensor B^{(i_n)},
\end{align*}
for all $x_1,\dots,x_k\in B^{(1)}$, where $\chi$ is the graded sign coming from exchanging the $x_i$'s.
This in turn means that 
\begin{align}
	\label{Eq:BialId}
	\tfrac{(e^{(1)})^{\tensor k}}{k!}\circ\Delta^{k-1}(x_1\cdots x_k)
	=
	\tfrac{1}{k!}\sum_{\sigma\in S_k}\chi(\sigma) x_{\sigma(1)}\tensor\cdots \tensor x_{\sigma(k)},
\end{align}
for all $x_1,\dots,x_k\in B^{(1)}$. 
\begin{theorem}\label{Thm: PBW?}
	The maps
	\begin{align*}
		I_k\colon \Sym^kB^{(1)}\ni x_1\vee\cdots\vee x_k\mapsto x_1\cdots x_k \in B^{(k)}
	\end{align*}
	are  isomorphisms for all $k\in \mathbb{N}_0$ and their sum is an isomorphism of commutative algebras. 
\end{theorem}
\begin{proof}
	Note that from the discussion above it
    is clear that the domains of the maps $I_k$ are in
    fact $B^{(k)}$.  We give an inverse to all of these
    maps:
	\begin{align*}
		J_k \colon B^{(k)}\ni x \longmapsto \tfrac{1}{k!}e^{(1)}(x_{(1)})\vee \cdots \vee e^{(1)}(x_{(k)})\in \Sym^k(B^{(1)}),
	\end{align*}
	where we use Sweedler's notation. Note that we have
    $I_k\circ J_k =e^{(k)}\at{B^{(k)}}=\id$. The identity
    $J_k\circ I_k=\id$ follows immediately from Equation
    \eqref{Eq:BialId}. Moreover, the map $I=\sum_k
    I_k\colon \Sym(B^{(1)})\to B$ is by definition an
    algebra map.
\end{proof}
\begin{remark}
	Note that Theorem \ref{Thm: PBW?} together with
    Equation \eqref{Eq:CoProdGrading} implies that the
    graded coalgebra associated to the filtration from
    \eqref{Def:filtration} is isomorphic to $\Sym(B^{(1)})$
    with the shuffle coproduct.
\end{remark}
Let us denote by $\Delta_L$ the graded cocommutator, i.e. 
\begin{align*}
	\Delta_L:=(\id-\tau)\circ \Delta,
\end{align*}
where $\tau\colon B\tensor B\to B\tensor B$ is the graded
    tensor flip. This means, using Equation
    \eqref{Eq:CoProdGrading}, that
\begin{align*}
	\Delta_L(B^{(\geq 1)})\subseteq B^{(\geq 1)}\tensor B^{(\geq 1)}.
\end{align*}
This means in particular that $B^{(\geq 1)}$ is a Lie
    sub-coalgebra of $B$. Moreover, one can check along the same
    lines that $B^{(\geq 2)}$ is a Lie coideal in $B^{(\geq 1)}$,
    so $B^{(1)}\cong B^{(\geq 1)}/B^{(\geq 2)}$ is a graded Lie
    coalgebra.  Note that we could define the associated Lie
    coalgebra also without using the splitting induced by the
    Eulerian idempotents, since by Theorem 1.3 $B^{(\geq 2)}$ is
    given by the image of the multiplication restricted to
    $B^{(\geq 1)}$, i.e.
\begin{align*}
	B^{(\geq 2)}=B^{(\geq 1)}\cdot B^{(\geq 1)}
\end{align*}
and $B^{(\geq 1)}=\ker \epsilon$. 	
Let us now consider the tensor coalgebra $T^c(V)$ of a graded
    vector space $V$. The coproduct is the deconcatenation
    coproduct $\Delta$:
\begin{align*}
	\Delta(v_1\cdots v_k)=\sum_{i=0}^{k} v_1\cdots v_i\tensor v_{i+1}\cdots v_k.
\end{align*}
Moreover, the graded commutative product 
\begin{align*}
	v_1\cdots v_l\tensor v_{l+1}\cdots v_k\mapsto \sum_{\sigma \in \mathrm{Sh}(l,k-l)}\chi(\sigma)v_{\sigma^{-1}(1)}
	\cdots v_{\sigma^{-1}(k)}.
\end{align*}
together with the canonical maps $\iota\colon \Bbbk\to T^c(V)$
    and $\epsilon\colon T^c(V)\to \Bbbk$, turns $T^c(V)$ into a
    graded commutative conilpotent bialgebra. Let us denote
    $\mathrm{coLie}(V):=e^{(1)}(T^c(V))$. Moreover, we have the
    canonical map
\begin{align*}
	\colie(V)\to V
\end{align*}
which is induced by the canonical map $\cc{T}^c(V)\to V$. It
    is known that $\colie(V)$ is the cofree conilpotent
    graded Lie coalgebra cogenerated by $V$, see \cite{dolgushevetal:2007}.
    This means in particular, using Theorem \ref{Thm: PBW?}, that
    the map
\begin{equation}\label{Eq:PBW}
	I\colon \Sym(\colie(V))\ni a_1\vee\cdots\vee a_k\to 
	a_1\cdots  a_k\in T^c(V)
\end{equation}
is an isomorphism of graded algebras and an isomorphism on the
    level of graded objects with respect to the corresponding
    filtration. Note that since $T^c(V)$ is the universal enveloping 
    conilpotent coalgebra of $\colie(V)$, this isomorphism is 
     nothing  else but the dual version of the Poincaré-Birkoff-Witt 
     Theorem for the cofree Lie coalgebra.

\subsection{Chevalley-Eilenberg cohomology of cofree Lie coalgebras}

Let $V$ be a graded vector space and let $\colie(V)$ be its
    cofree conilpotent coalgebra. We consider the graded vector
    space
\begin{align}
	T^c(V)\tensor \Sym(\colie(V)[-1])
\end{align}
together with the differential
\begin{align}
	\D (C\tensor \xi)
	=(-1)^{|C_{(1)}|}C_{(1)}\tensor\epsilon(C_{(2)})\wedge \xi
	+(-1)^{|C|}C\tensor \delta\xi
\end{align}
where $\epsilon$ is the canonical projection map $T^c(V)\to \colie(V)$ and the
comultiplication is extended as a degree $+1$ derivation of
the product in $\Sym (\colie(V)[-1)]$.  As a first step, we use
the isomorphism $I$ from Equation \ref{Eq:PBW} to identify 
\begin{align}
	T^c(V)\tensor \Sym^\bullet (\colie(V)[-1])\cong 
	\Sym(\colie(V))\tensor \Sym(\colie(V)[-1])
\end{align}
and introduce the differential
\begin{align}
	\D_\mathrm{ab}(\xi_1\vee\dots\vee\xi_k\tensor \alpha)
	=
	\sum_{i=1}^k (-1)^{\sum_{j\neq i}|\xi_j| + |\xi_i|\sum_{l=i+1}^k |\xi_l|}\xi_1\vee\overset{\wedge_i}
	{\cdots}\vee\xi_k\tensor\xi_i\wedge \alpha,
\end{align}
where $\wedge_i$ denotes the omission of the $\xi_i$. Note
that this is the usual (graded) de Rham differential, which
has a canonical homotopy $h_\mathrm{ab}\colon
\Sym^c\colie(V)\tensor \Sym^\bullet\colie(V)[-1]
[\to\Sym^c\colie(V) \tensor \Sym^{\bullet-1} \colie(V)[-1]$
given by
\begin{align*}
	h_\mathrm{ab}(\xi_1\vee \dots \vee \xi_l\tensor \alpha_1\wedge\dots \wedge\alpha_k)
	=
	\frac{(-1)^{\sum_{i=1}^\ell|\xi_i|}}{\ell+k}\sum _{i=1}^k( -1)^{|\alpha_i|\sum_{j=1}^{i-1}|\alpha_j|}
	\xi_1\vee\dots\vee\xi_\ell\vee \alpha_i\tensor\alpha_1\wedge\overset{\wedge_i}{\dots} \wedge \alpha_k.
\end{align*}
Note that $h_\mathrm{ab}$ is just the graded version of the canonical homotopy constructed in the 
Poincaré-Lemma for polynomial differential forms. 
A short computation shows that the following equality holds:
\begin{align}
	\D_\mathrm{ab}h_\mathrm{ab}+h_\mathrm{ab}\D_\mathrm{ab}
	=
	\id-\sigma,
\end{align}
where $\sigma$ denotes the projection to symmetric and anti-
symmetric degree $0$. The idea is now, that using the
Poincaré-Birkoff-Witt isomorphism, we can see $\D$ as a
perturbation of $\D_\mathrm{ab}$, so for $b:=\D-\D_\mathrm{ab}$,
we get
\begin{align}
	b\colon \Sym(\colie(V))\tensor
	\Sym^\bullet(\colie(V)[-1])\to  
	\Sym(\colie(V))\tensor
	\Sym^{\bullet+1}(\colie(V)[-1]].
\end{align}

We want to apply Prop.~\ref{prop:HomologicalPerturbation} to get a
    deformation retract for the differential
    $\D=\D_\mathrm{ab}+b$.  First, observe that the map
\begin{align}
	bh_{\mathrm{ab}}\colon \Sym(\colie(V))\tensor
	\Sym^\bullet(\colie(V)[-1])
	\to \Sym(\colie(V))\tensor
	\Sym^\bullet(\colie(V)[-1]),
\end{align}
is locally nilpotent. In facts, the map $bh_{\mathrm{ab}}$
    raises the symmetric degree in the first tensor factor by at
    least one which can be seen by Equation
    \eqref{Eq:CoProdGrading}. Moreover, the total tensor degree is
    not changed, thus $bh_{\mathrm{ab}}$ is locally
    nilpotent. As an immediate consequence we have:
\begin{theorem}\label{Thm:Liecoalgcohomology}
	Let $V$ be a graded vector space, 
	then the map
	\begin{align}
		h= h_{\mathrm{ab}}
		\sum_{i=0}^\infty (-bh_{\mathrm{ab}})^{i}\colon  
		\Sym(\colie(V))\tensor
		\Sym^\bullet(\colie(V)[-1])
		\subseteq \Sym(\colie(V))\tensor
		\Sym^{\bullet+1}(\colie(V)[-1])
	\end{align} 
	is a well-defined map of degree $-1$ and a canonical homotopy for the 
	differential $\D$, 
	i.e. $\D h+h\D=\id-\sigma$. Moreover, 
	we have $h^2=0$ and $h(1)=0$. 
\end{theorem}    
With this we want to compute the Chevalley--Eilenberg cohomology
without coefficients of the cofree conilpotent Lie coalgebra
$\colie(V)$ cogenerated by a graded vector space $V$. Let us define the complex
\begin{align}
	(\Bbbk \oplus V[-1])\tensor T^cV\cong T^cV\oplus V[-1]\tensor T^cV
\end{align}
with the differential 
$\partial\colon T^cV\oplus V[-1]\tensor T^cV\to T^cV\oplus 
V[-1]\tensor T^cV$ defined by 
\begin{align*}
	\partial (a_1\tensor\dots\tensor a_k,b)=
	(0,a_1\tensor(a_2\tensor\dots\tensor a_k)). 
\end{align*}
Note that there is a canonical homotopy $k\colon T^cV\oplus
    V[-1]\tensor T^cV\to T^cV\oplus V[-1]\tensor T^cV$ given by
\begin{align}
	k(a, b_0\tensor(b_1\tensor\dots\tensor b_k))=
	(b_0\tensor\dots\tensor b_k,0),
\end{align}
which fulfills $(\partial k + k\partial) (a,b)=(\pr_\Bbbk(a),0)$.
The idea is now that the
complex
\begin{align}
	(\Bbbk\oplus V[-1])\tensor T^cV \tensor \Sym(\colie(V)[-1])
\end{align}
has two differentials: $\partial$ on the first two tensor factors and
$\D$ on the last two tensor factors (with Koszul signs).  It is easy
to check that they commute and therefore $D:=\partial+\D$ is a
differential as well. This means we get two deformation retracts:
\begin{center}
	\begin{tikzcd}
		\Bbbk\oplus V[-1] \arrow[rr, shift left, "i"]  &  &
		\bigg((\Bbbk\oplus V[-1])\tensor T^cV \tensor \Sym(\colie(V)[-1]) , \D\bigg)\arrow[ll, shift left, "p"] 
		\arrow[loop right, 
		distance=4em, start anchor={[yshift=1ex]east}, end anchor={[yshift=-1ex]east}]{}{h}        \end{tikzcd}
\end{center}
and 
\begin{center}
	\begin{tikzcd}
		\Sym (\colie(V)[-1] ) \arrow[rr, shift left, "j"]&  &
		\bigg((\Bbbk\oplus V[-1])\tensor T^cV \tensor \Sym(\colie(V)[-1]) , \partial\bigg)\arrow[ll, shift left, "q"] \arrow[loop right, 
		distance=4em, start anchor={[yshift=1ex]east}, end anchor={[yshift=-1ex]east}]{}{k}  
	\end{tikzcd}
\end{center}
Now we use Prop.~\ref{prop:HomologicalPerturbation} to perturb the first
    deformation retract by $\partial$ and the second one by
    $\D$. One can check, that in this way one obtains a
    deformation retract
\begin{center}
	\begin{tikzcd}
		\label{DefRetract:cofreecolie}
		V[-1]\oplus \Bbbk \arrow[rr, shift left, "i"]  &  &
		(\Sym(\colie(V)[-1]) , \delta)\arrow[ll, 
		shift left, "p"]\arrow[loop right, 
		distance=4em, start anchor={[yshift=1ex]east}, end anchor={[yshift=-1ex]east}]{}{H}  
	\end{tikzcd}
\end{center}
where $i$ and $p$ are the canonical maps $i(v)=(v)$ and
$p(v_1|\dots|v_k)=v_1$ for $k=1$ and vanishes otherwise.  Note
that the homotopy $H$ is by construction just a combinatorial
rearranging of the corresponding tensor factors in $V$ and
their quotients and symmetrization while not changing the
total amount of tensor factors, therefore this construction is
natural in the vector space $V$.

This implies immediately $H\circ i=0$ and $p\circ H=0$ by
counting tensor degrees. Even though we cannot ensure $H^2=0$,
we just replace it by $H\circ \delta\circ H$ to achieve that
the homotopy squares to zero. And we obtain a special
deformation retract, which restricts to
\begin{center}
	\begin{tikzcd}
		\label{DefRetract:cofreecolie}
		V[-1] \arrow[rr, shift left, "i"]  &  &
		(\comm(\colie(V)[-1]) , \delta)\arrow[ll, 
		shift left, "p"] \arrow[loop right, 
		distance=4em, start anchor={[yshift=1ex]east}, end anchor={[yshift=-1ex]east}]{}{H}  
	\end{tikzcd}
\end{center} 
Let us now assume that $V$ is endowed with a differential 
$\D_V$, which we can extend to $\colie(V)$ by 
\begin{align*}
	\D_V(v_1|\dots|v_k)=
	\sum_{i} \pm (v_1|\dots|\D_V v_i|\dots|v_k)
\end{align*}
to get a differential graded Lie coalgebra.  We can extend it
further to a derivation of the symmetric product. Since
$H^2=0$ and because of the naturality of $H$, we have
$H\D_V+\D_V H=0$, and we obtain 

\begin{proposition}
\label{prop:a6}
Let $(V,\D_V)$ be a cochain complex, then 
\begin{equation}
	\begin{tikzcd}
		\label{DefRetract:cofreecolie}             (V, \D_V) \arrow[rr, shift left, "i"]  &  &             (\comm(\colie(V[1])[-1]) , \delta+\D_V)\arrow[ll,              shift left, "p"] \arrow[loop right,  	distance=4em, start anchor={[yshift=1ex]east}, end anchor={[yshift=-1ex]east}]{}{H} \end{tikzcd} 
\end{equation}
is a special deformation retract. Moreover, this construction is natural in the sense that if $f\colon (V,\D_V)\to (W,\D_W)$ is a cochain map, then the induced map 
$F\colon \comm(\colie(V[1])\to \comm(\colie(W[1])$ induces a morphism of special deformation retracts. 
\end{proposition}
\begin{proof}
This is just an application of Prop.~\ref{prop:HomologicalPerturbation} where it is easy to check that the maps $i$ and $p$ do not change while
perturbing..
\end{proof}
Assuming now that we start with a dg commutative algebra $(A,
\D_A)$, we can define the map $d\colon \colie(A[1])\to
\colie(A[1])$ by
\begin{align*}
	d(a_1|\dots|a_k)
	=
	\sum_{i=1}^{k-1}(-1)^i
	(a_{1}|\dots| a_ia_{i+1}|\dots|a_{k}).
\end{align*}
Note that $D_A=\D_A+d$ define a differential on $
\colie(A[1])$,being a coderivation with respect to the cofree
cobracket. This means one can extend it to a differential to
$\Sym(\colie(A[1])[-1]) $ to obtain a dg commutative algebra
with respect to the differential $\delta+D_A$. Note that
$\cc{\Sym}(\colie(A[1])[-1])$ is a dg commutative subalgebra
of $\Sym(\colie(A[1])[-1])$.  Using again Prop.~\ref{prop:HomologicalPerturbation}, 
one obtains 
\begin{proposition}
Let $A$ be a commutative dg algebra then  
\begin{equation}
	\begin{tikzcd}
		\label{DefRetract:UCommAlg}
		(A\oplus\Bbbk,\D_A) \arrow[rr, shift left, "i"]  &  &
		(\Sym(\colie(A[1])[-1]) , \delta+D_A)\arrow[ll, 
		shift left, "P"] \arrow[loop right, 
		distance=4em, start anchor={[yshift=1ex]east}, end anchor={[yshift=-1ex]east}]{}{H_A}  
	\end{tikzcd}
\end{equation}
and
\begin{equation}
	\begin{tikzcd}\
		\label{DefRetract:CommAlg}
		(A,\D_A) \arrow[rr, shift left, "i"]  &  &
		(\comm(\colie(A[1])[-1]) , \delta+D_A)\arrow[ll, 
		shift left, "P"]\arrow[loop right, 
		distance=4em, start anchor={[yshift=1ex]east}, end anchor={[yshift=-1ex]east}]{}{H_A}  
	\end{tikzcd}
\end{equation}
are special deformation retracts where
$P$ is now given by 
\begin{align*}
	P((a_1)\vee\cdots\vee(a_k))=a_1\cdots a_k.
\end{align*}
Moreover, this construction is natural in the sense that a morphism of dg commutative algebras induces a morphism of the corresponding special deformation retracts. 
\end{proposition}

\begin{proof}
The deformation retract  is again an easy application of Proposition ~\ref{prop:HomologicalPerturbation}
Note that $P$ is given by $P((a_1)\vee\cdots\vee(a_k))=a_1\cdots a_k$ 
is an easy observation counting tensor
degrees, since one can check fairly easily that
\begin{equation}
	\label{Eq: Homotopy}
	H((a_1)\vee\cdots\vee (a_k))=\frac{1}{k(k-1)}\sum_{i=1}^k \sum_{j=1,j\neq i}^k 
	\pm (a_i|a_j)\vee (a_1)\vee \overset{ij}{\dots}\vee (a_k).  
\end{equation}
\end{proof}
This means in particular that $P$ is a quasi-isomorphism of dg
    commutative algebras, where the dg commutative algebra
    structure of the first retract is given by $\D_A(1)=0$ and
    $1\cdot a=a$. For the purposes of this paper we need to 
     include an $A$-module $\mathscr{M}$.
    Let us stick to the case of a commutative algebra
    concentrated in degree $0$: we consider the dg Lie coalgebra
\begin{align*}
	\colie(A[1])\oplus T^cA[1]\tensor \mathscr{M}
\end{align*}
where the differential is given as a direct sum of $\D \colon \colie(A[1])\to \colie(A[1])$ and the differential $\D_\mathscr{M}\colon 
T^cA[1]\tensor \mathscr{M}\to T^cA[1]\tensor \mathscr{M}$ given by 
\begin{align}
	\D_\mathscr{M}(a_1\tensor\cdots\tensor a_k\tensor m)=
	\D(a_1\tensor \cdots \tensor a_k)\tensor m +(-1)^{k+1}
	a_1\tensor\cdots \tensor a_{k-1}\tensor a_k m. 
\end{align}
Note that by construction it is clear that the sum of the
differentials turns the $\colie(A[1])\oplus T^cA[1]\tensor
\mathscr{M}$ into a dg Lie coalgebra. Let us consider its (reduced)
Chevalley--Eilenberg complex
\begin{align*}
	\comm(\colie(A[1])\oplus T^cA[1]\tensor \mathscr{M}[-1])\cong 
	\comm(\colie(A[1]) \oplus \Sym(\colie(A[1])\tensor 
	\comm(T^cA[1]\tensor \mathscr M[-1]).
\end{align*}
Here the splitting is a splitting of cochain complexes. 
\begin{proposition}
	\label{prop:a8}
Let $A$ be a unital commutative algebra and let $\mathscr{M}$ be a projective $A$-module. 
For each projective basis $\{e_i\}_{i\in I}$ with dual $\{e^i\}_{i\in I}$ we can construct 
a deformation retract 
\begin{equation} \label{DefRetract: projMod}
	\begin{tikzcd}
		\Sym_{A}\mathscr{M}[-1] 
		\arrow[rr, shift left, "I"]  &  &
		(\comm(\colie(A[1])[-1]\oplus T^cA[1]\tensor \mathscr{M}[-1]), \delta+D_A+D_\mathscr{M})\arrow[ll, 
		shift left, "\Pi"] 
		\arrow[out=240,in=300,loop,swap].
	\end{tikzcd}
\end{equation}
with 
	\begin{align*}
	\Pi(x_1\vee \dots x_k)= p(x_1)\wedge \cdots \wedge p(x_k),
	\end{align*}
for $x_i\in \colie(A[1])[-1]\oplus T^cA[1]\tensor \mathscr{M}[-1]$, where $p((a))=a$ for 
$(a)\in \colie^1(A[1])$ and $p(1\tensor m)=m$ and vanishes identically for all other elements. And $I$ is given by 
	\begin{align*}
	I(m_1\wedge \cdots \wedge m_k)=(e^{i_k}(m_1)\cdots e^{i_k}(m_k))\vee(e_{i_1})\vee\cdots    
	 \vee(e_{i_k}).
	\end{align*}
Moreover, the construction is natural in the sense that if $f\colon B\to A $ is an algebra morphism and 
we consider the $B$-module $\mathscr{M}_B=B\tensor_A \mathscr{M}$ then everything extends to a morphism of deformation 
retracts in the suitable sense. 
\end{proposition}

\begin{proof}
The
    cohomology of the first summand is given by $A$, see the
    deformation complex \eqref{DefRetract:CommAlg}. The
    differential of the second summand is given by the sum of
    $\delta+D_A$ on the first tensor factor, on the the extension
    of $\D_\mathscr{M}$ as a degree $+1$ derivation and the
    symmetric extension of the coaction $\rho\colon T^cA[1]\tensor
    \mathscr{M} \to \colie(A[1])\tensor T^cA[1] \tensor\mathscr{M}
    $.  Let us summarize the latter two as $D_\mathscr{M}$. We
    consider the deformation retract
\begin{equation}
	\begin{tikzcd}
		(A\oplus\Bbbk)\tensor  \comm(T^cA[1]\tensor \mathscr M)\arrow[rr, shift left, "i"]  &  &
		(\Sym(\colie(A[1])[-1])\tensor  \cc{\Sym}(T^cA[1]\tensor \mathscr M[-1]) , \delta+D_A)\arrow[ll, 
		shift left, "P"] \arrow[out=240,in=300,loop,swap,  "H_A"],
	\end{tikzcd}
\end{equation}
which is just a tensor product of $T^cA[1]\tensor \mathscr{M}$
    with the deformation retract \eqref{DefRetract:UCommAlg}. We
    want to perturb this deformation retract now with
    $D_\mathscr{M}$.  Note that $H_A$ operates in the first tensor
    factor and $D_\mathscr M$ lowers the total tensor degree in
    the second factor by $1$ and thus $H_A\circ D_\mathscr{M}$ and
    $D_\mathscr{M}\circ H_A$ are locally nilpotent and we can
    apply the homological perturbation lemma. We therefore obtain
    the deformation retract
\begin{equation}
	\begin{tikzcd}
		(A\oplus\Bbbk)\tensor  \comm(T^cA[1]\tensor \mathscr M[-1]),
		b) \arrow[rr, shift left, "\widetilde{i}"]  &  &
		(\Sym(\colie(A[1])[-1])\tensor  \cc{\Sym}
		(T^cA[1]\tensor \mathscr M) , 
		\delta+D_A+D_\mathscr{M})\arrow[ll, 
		shift left, "\widetilde{P}"] 
		\arrow[out=240,in=300,loop,swap,  "\widetilde{H}_A"],
	\end{tikzcd}
\end{equation}
where the maps $\widetilde{i}$ and $\widetilde{P}$ as well as
    the perturbation of the differential are given by the usual
    formulas.  Again, by counting degrees tensor degrees, one can
    check that
\begin{align*}
	\widetilde{P}=P \ \text{ and } \ \widetilde{i}=i
\end{align*}
Let us take a closer look at the differential $b$, which is given by
\begin{align*}   
	b((a+r)\tensor a_1\tensor\cdots a_k \tensor m)=&
	(aa_1+ra_1)\tensor a_1\tensor\cdots a_k \tensor m\\&
	+\sum_{i=1}^{k}(-1)^{i}(a+r)\tensor a_1\tensor\dots 
	\tensor a_ia_{i+1}\tensor\dots \tensor a_k\tensor m\\&
	+(-1)^{k}(a+r)\tensor a_1\tensor\cdots a_k m)
\end{align*}
and its symmetric extension to $A\oplus\Bbbk)\tensor  
\comm(T^cA[1]\tensor \mathscr M$. Note that $b$ is $A\oplus\Bbbk$-
linear for the obvious action of $A\oplus\Bbbk$ on  
$(A\oplus\Bbbk)\tensor  \comm(T^cA[1]\tensor \mathscr M)$ and thus can be 
regarded as a differential on 
\begin{align*}
	\cc{\Sym}_{A\oplus\Bbbk}((A\oplus\Bbbk)\tensor  T^cA[1]\tensor \mathscr M)
	\cong 
	(A\oplus\Bbbk)\tensor  \comm(T^cA[1]\tensor \mathscr{M})
\end{align*}
Since $\mathscr{M}$ is projective we can find 
 a projective basis of $\mathscr{M}$ $\{e_{i}\}_{i\in I }$ 
together with a $A$-linear dual $\{e^i\}_{i\in I}$, i.e. every 
element $m\in \mathscr{M}$ can be written as $m=e^i(m)e_i$, where 
we use Einstein summation convention. Note that $e^i(m)$ 
is only non zero for a finite number of $i$'s. We define the 
operator 
\begin{align*}
	\widetilde{h}_A((a+r)\tensor a_1\tensor\dots\tensor a_k\tensor m)=
	(-1)^k(a+r)\tensor a_1\tensor\dots\tensor a_k\tensor e^i(m)\tensor e_i
\end{align*}
and the maps $\pi\colon (A\oplus\Bbbk)\tensor
\mathscr{M}\ni((a+r),m)\mapsto am+rm\in \mathscr{M}$ and $\iota\colon
\mathscr{M}\ni m \mapsto ((e^i(m),0),e_i)\in (A\oplus\Bbbk)\tensor
\mathscr{M}$. Note that $\pi$, $\iota$ and $h_A$ are left
$A\oplus\Bbbk$-linear. Moreover, we have that
\begin{center}
	\begin{tikzcd}
		\mathscr{M}[-1] 
		\arrow[rr, shift left, "\iota"]  &  &
		((A\oplus\Bbbk)\tensor  
		T^cA[1]\tensor \mathscr{M}[-1] , 
		b)\arrow[ll, 
		shift left, "\pi"] 
		\arrow[out=240,in=300,loop,swap,  "\widetilde{h}_A"],
	\end{tikzcd}
\end{center}
is a deformation retract and all the maps are $A\oplus\Bbbk$-
linear. Note that we have $\pi\circ \widetilde{h}_A=0$, but geberally
we have $\widetilde{h_A}\circ \iota\neq0\neq
\widetilde{h}_A^2$. Nevertheless, we can achieve that by modifiying
$\widetilde{h_A}$ to a new $A\oplus\Bbbk$-linear homotopy $h_A$ to
obtain a special deformation retract
\begin{center}
	\begin{tikzcd}
		\mathscr{M}[-1] 
		\arrow[rr, shift left, "\iota"]  &  &
		((A\oplus\Bbbk)\tensor  
		T^cA[1]\tensor \mathscr{M}[-1] , 
		b)\arrow[ll, 
		shift left, "\pi"] 
		\arrow[out=240,in=300,loop,swap,  "h_A"],
	\end{tikzcd}
\end{center}
Using the formula of the symmetric tensor trick (see \cite[Proposition 4.2]{kraft:2024}), we can obtain a special deformation retract
\begin{center}
	\begin{tikzcd}
		\cc{\Sym}_{A}\mathscr{M}[-1] 
		\arrow[rr, shift left, "\cc{\Sym}_{A}(\iota)"]  &  &
		(\cc{\Sym}_{A}((A\oplus\Bbbk)\tensor  
		T^cA[1]\tensor \mathscr{M}[-1] ), 
		b)\arrow[ll, 
		shift left, "\cc{\Sym}_{A}(\pi)"] 
		\arrow[out=240,in=300,loop,swap,  "L_A"],
	\end{tikzcd}
\end{center}
Note that we canonically have
\begin{align*}
	\cc{\Sym}_{A}((A\oplus\Bbbk)\tensor  
	T^cA[1]\tensor \mathscr{M}[-1]
	\cong 
	(A\oplus \Bbbk)\tensor \comm(T^cA[1]\tensor \mathscr{M}),
\end{align*}
so we have by concatenation 
\begin{center}
	\begin{tikzcd}
		A\oplus \cc{\Sym}_{A}\mathscr{M}[-1] 
		\arrow[rr, shift left, "I"]  &  &
		(\comm(\colie(A[1])\oplus T^cA[1]\tensor \mathscr{M}[-1]), \delta+D_A+D_\mathscr{M})\arrow[ll, 
		shift left, "\Pi"] 
		\arrow[out=240,in=300,loop,swap].
	\end{tikzcd}
\end{center}
The maps $\Pi$ and $I$ are canonically given as in the statement of the proposition. 
\end{proof}
\end{appendix}

{
	\footnotesize
	\renewcommand{\arraystretch}{0.5}

}

\end{document}